\documentclass[a4paper,11pt]{amsart}

\usepackage{amssymb,amsfonts,amsthm,amscd,stmaryrd}

\usepackage{ae}
\usepackage[utf8]{inputenc}
\numberwithin{figure}{section}

\usepackage{mathrsfs,a4wide}
\usepackage{esint}

\usepackage{color}

\usepackage{import}

\usepackage[leqno]{amsmath}

\numberwithin{figure}{section}

\newtheorem{theorem}{Theorem}[section]
\newtheorem{lemma}[theorem]{Lemma}
\newtheorem{proposition}[theorem]{Proposition}

\theoremstyle{definition}
\newtheorem{definition}[theorem]{Definition}

\newtheorem{remark}[theorem]{Remark}
\numberwithin{equation}{section}
\def\rightangle{\vcenter{\hsize5.5pt
    \hbox to5.5pt{\vrule height7pt\hfill}
    \hrule}}
\newcommand{\rtangle}{\mathrel{\rightangle}}
\newcommand{\wto}{\rightharpoonup}
\newcommand{\de}{\delta}
\newcommand{\R}{\mathbb{R}}
\newcommand{\N}{\mathbb{N}}
\newcommand{\Z}{\mathbb{Z}}
\newcommand{\Ha}{\mathcal{H}}
\newcommand{\T}{\mathbb{T}}
\newcommand{\beq}{\begin{equation}}
\newcommand{\eeq}{\end{equation}}
\renewcommand{\P}{P_{\T^3}}
\newcommand{\dist}{{\rm dist}}
\newcommand{\um}{\"}
\newcommand{\eps}{\varepsilon}
\newcommand{\e}{\varepsilon}
\newcommand{\vphi}{\varphi}

\newcommand{\diver}{\operatorname{div}}
\newcommand{\Div}{\operatorname{div}}
\newcommand{\pa}{\partial}
\newcommand{\Ht}{\widetilde H}
\newcommand{\Om}{\Omega}
\newcommand{\medint}{-\kern -,375cm\int}
\newcommand{\medintinrigo}{-\kern -,315cm\int}
\begin{document}

\title[Mullins-Sekerka and surface diffusion flows]{Nonlinear stability results for  the modified Mullins-Sekerka  and the surface diffusion flow}

\author{E. Acerbi, N. Fusco, V. Julin, M. Morini}
\address[E.\ Acerbi]{Dipartimento di Matematica, Universit\`a di Parma, Italy}
\email{emilio.acerbi@unipr.it}
\address[N.\ Fusco]{Dipartimento di Matematica e Applicazioni,
Universit\`{a} di Napoli ``Federico II'', Italy}
\email{n.fusco@unina.it}
\address[V.\ Julin]{Matematiikan ja tilastotieteen laitos,  Jyv\"as\-kyl\"an Yliopisto, 
         Finland}
\email{vesa.julin@jyu.fi}
\address[M.\ Morini]{Dipartimento di Matematica, Universit\`a  di Parma, Italy}
\email{massimiliano.morini@unipr.it}


\begin{abstract}  It is shown that any three-dimensional periodic configuration that is strictly stable  for the area functional is exponentially stable for the surface diffusion flow and for the Mullins-Sekerka or Hele-Shaw flow. The same result holds for 
three-dimensional periodic configurations that are strictly stable with respect to the sharp-interface Ohta-Kawaski energy. In this case, they are exponentially stable for the so-called modified Mullins-Sekerka flow.
\end{abstract}

\maketitle
\tableofcontents

\section{Introduction}
In this paper we establish new  global-in-time  existence and long-time behavior results in three-space dimensions for two physically relevant geometric motions; namely,  the (modified) Mullins-Sekerka and the surface diffusion flows.   Let $\Om$ be a bounded open set of $\R^N$. We start by recalling that  a smooth flow of sets $(E_t)_{t}\subset\subset\Omega$,   defined on some (maximal) time interval $(0, T^*)$, is a solution of  the (modified) Mullins-Sekerka flow if the evolution is governed by the following law
\beq\label{MSintro}
\begin{cases}
V_t= [\pa_{\nu_t}w_t] & \text{on $\pa E_t$,}\\
\Delta w_t=0 & \text{in $\Om\setminus \pa E_t$,}\\
w_t=H_t+4\gamma v_t & \text{on $\pa E_t$,}\\
-\Delta v_t=u_{E_t}-\medintinrigo_{\Om}u_{E_t}\,, & \text{in $\Om$,}
\end{cases}
\eeq
where both $w_t$ and $v_t$ are subject to homogeneous Neumann boundary conditions on $\pa \Om$ or to {\em periodic boundary conditions} in the case $\Om=\T^N$, with $\T^N$ denoting the $N$-dimensional flat torus. Here and in the following $V_t$ stands for the outer normal velocity of the moving boundary $\pa E_t$, $H_t$ denotes the mean curvature of $\pa E_t$, $\gamma\geq 0$ is a fixed parameter, $u_{E_t}:=2\chi_{E_t} -1$ and $[\pa_{\nu_t}w_t]$ is a short notation for  the jump of the normal derivative of $w_t$ at $\pa E_t$; more precisely, $[\pa_{\nu_t}w_t]:=\pa_{\nu_t}w^+_t-\pa_{\nu_t}w^-_t$, with  $w^+_t$ and $w^-_t$ denoting the restrictions of $w_t$ to $\Om\setminus E_t$ and $E_t$, respectively.  In the case $\gamma=0$ the {\em potential} $v_t$ becomes irrelevant and we recover the classical Mullins-Sekerka flow (see \cite{MS}), which is also sometimes referred to as the  {\em two-phase Hele-Shaw flow with surface tension} (see \cite{ES97}). Such models arise as singular limits of the Cahn-Hilliard equation in the case $\gamma=0$, as formally derived in \cite{Pego} and then rigorously proved in \cite{AlBaCh}, and of a modified (nonlocal) version of the Cahn-Hilliard equation in the case $\gamma>0$. Such a modified equation has been proposed in \cite{OK} to describe phase separation in diblock copolymer melts and its convergence to \eqref{MSintro}  has been established in \cite{Le}. Under Neumann boundary conditions if $\gamma=0$ and  $(E_t)_{t}\subset\!\subset\Omega$ then  Alexandrov's Theorem implies that the only possible equilibria for \eqref{MSintro}
are union of balls.   On the contrast, in the periodic case  or when $\gamma>0$  the sets of equilibria has a  much  richer structure as we will see below. 
 
 The second geometric flow we are dealing with is the motion of sets  {\em  by surface diffusion}; in this case the evolution of $E_t$ is governed by the law
 \beq\label{SDintro}
 V_t=\Delta_\tau H_t \qquad\text{on $\pa E_t$,}
 \eeq
where $\Delta_\tau$  denotes the surface Laplacian or Laplace-Beltrami operator on $\pa E_t$.  Such a law has been proposed in the physical literature to describe the evolution of interfaces between solid phases driven by surface diffusion of atoms under the action of a chemical potential (see for instance \cite{GurJab} and the references therein).  

The two flows share several features: they are both {\em volume preserving} and may be regarded as  suitable gradient flows of the  (nonlocal) area functional (also known as sharp-interface Ohta-Kawasaki energy):
\beq\label{Jintro}
J(E):= P_\Om(E)+ \gamma \int_{\Om}\int_\Om G(x,y)u_E(x)u_E(y)\, dxdy\,,
\eeq
where $P_\Om$ is the standard perimeter (or area) functional in $\Om$, while $G$ stands for the Green's function in $\Om$ and $u_E:=2\chi_E-1$.  More precisely, \eqref{MSintro} can be  seen as the gradient flow of \eqref{Jintro} with respect to a suitable $H^{-\frac12}$-Riemannian structure (see for instance \cite{Le}) formally defined on the space of shapes, while \eqref{SDintro} is the gradient flow of the area function, that is of \eqref{Jintro} with $\gamma=0$, with respect to a $H^{-1}$-type Riemannian structure  (see \cite{cahn-taylor94}).  
In contrast with the more standard mean curvature flow, one cannot expect a   comparison principle to hold for \eqref{MSintro} and 
\eqref{SDintro}. This makes  it very difficult to apply  weak methods such as those  based on the notion of viscosity solution.
 
Since in fact singularities (such as pinching) may form in finite time (see for instance \cite{BerBerWit, Mayer}),  as far  as smooth flows are concerned one can only expect  in general local-in-time existence and uniqueness: see \cite{Chen} and \cite {ES97, Prockert} for the Hele-Shaw model in the two-dimensional and the $n$-dimensional case, respectively, \cite{EscNis} for the modified Mullins-Sekerka flow, and \cite{ElGar} and \cite{EMS} for the motion by surface diffusion in two and higher dimensions, respectively. For a very weak (distributional) notion of global-in-time solution to the Mullins-Sekerka flow in three dimensions,  obtained via a minimizing movements approach, we refer to \cite{Ro}.
Finally,  we remark that, again in contrast with the motion by mean curvature, both \eqref{MSintro} and \eqref{SDintro} do not preserve convexity (see \cite{Ito, EsMa}).

The nonlocal area functional \eqref{Jintro} is the sharp-interface limit of the so-called $\e$-diffuse Ohta-Kawasaki energy, which was proposed in  \cite{OK} to model the behavior of a class of two-phase materials called {\em diblock coplymers}. From the mathematical point of view, the main new feature is the presence of a nonlocal Green's function term,   which acts as a long-range repulsive interaction of Coulombic type.  While the perimeter term favors  the formation of large connected regions of pure phases  with minimal interface area,  the double integral term prefers scattered configurations with  several tiny connected components that   try to separate from each other as much as possible, due to the repulsive nature of their interaction.  The two competing trends often lead to the formation of stable nontrivial patterns, with a rather complex structure.  We refer to \cite{MorRev} and the references therein for a review on the 
Ohta-Kawasaki energy and some  related mathematical results. 

We now describe the results of our paper. As already mentioned, we are interested in finding a class of initial data for which we can prove the existence of a  global-in-time solution and study its long-time behavior. We focus on the periodic setting in three-dimensions; that is, we take $\Om=\T^3$ in \eqref{MSintro} and \eqref{SDintro} and we assume spatial one-periodicity both on the evolving sets and the functions involved. In other words, finding a solution in $\T^3$ is equivalent to finding a solution in the whole space $\R^3$, which is one-periodic in space. All the results and arguments that we present clearly hold also for $N=2$.  However, for the sake of presentation we decided to stick to the physically relevant case $N=3$.  

Because of the gradient flow structure of the two flows, it is very natural to expect that if the initial set is sufficiently close to a stable critical point (or a local minimizer) $F$ of the energy functional $J$, then the flow exists for all times and asymptotically converges to $F$.  

The proper notion of criticality and stability can be defined  in terms of the first and second variation of the energy by a standard procedure that we recall in the following:  We say that a smooth subset $F\subset\T^3$ is {\em critical} for \eqref{Jintro} if for any (admissible) smooth one-parameter family of volume preserving diffeomorphisms $(\Phi_t)_t$ we have that $\frac{d}{dt} J(\Phi_t(F))\bigl|_{t=0}=0$. It turns out 
(see for instance \cite{ChSte}) that a smooth set $F$ is critical if and only if 
\beq\label{critintro}
H_{\pa F}+4\gamma v_F=\mathrm{constant} \qquad\text{on $\partial F$,}
\eeq
where $H_{\pa F}$ is the mean curvature of $\pa F$ and $v_F(\cdot):=\int_{\T^3}G(\cdot, y)(2\chi_F(y)-1)\, dy$ is the potential associated with  $F$ (see also \eqref{MSintro} where $v_t$ stands for $v_{F_t}$). When $\gamma=0$ one recovers the classical constant mean curvature condition. Next, given a critical set $F$ we may compute its  {\em second variation}:  By the results of \cite{ChSte} (see also \cite{AFM, JuPi, Mur0}), we  associate with it a  quadratic
form $\pa^2J(F)$ defined over all functions $\vphi\in \Ht(\pa F):=\{\vphi\in H^1(\pa F):\, \int_{\pa F}\vphi\, d\Ha^2=0 \}$. This quadratic form is related to the second variation of $J$ by the following equality
\beq\label{J2intro}
\frac{d^2}{dt^2} J(\Phi_t(F))\biggl|_{t=0}=\pa^2J(F)[X\cdot \nu]\,,
\eeq
where $X\cdot\nu$ is the (outer) normal component of the velocity field $X$ of $(\Phi_t)_t$ on $\pa F$. The expression of $\pa^2J(F)$ can be computed explicitly, see \eqref{J2}.  Note that the condition $\int_{\pa F}\vphi\, d\Ha^2=0$ is related to the fact that we allow only volume preserving variations.

 The notion of stability  amounts to requiring that $\pa^2J$ is positive definite in a suitable sense.  However, we have to take into account that $J$ is translation invariant, so that in particular $J(F)=J(F+t\eta)$ for all 
$\eta\in \R^3$ and $t\in \R$. By differentiating   twice this identity with respect to $t$, one obtains 
$\pa^2J(F)[\eta\cdot \nu]=0$, thus showing that  there is always a finite dimensional subspace of {\em infinitesimal translations}
\beq\label{TEintro}
T(\pa  F):=\{\vphi\in \Ht(\pa F) \, : \, \vphi = \eta\cdot \nu, \, \, \eta\in \R^3\}
\eeq
where the second variation vanishes. In view of these observations, we say that the critical set $F$ is {\em strictly stable} if
\beq\label{J2>0intro}
\pa^2J(F)[\vphi]>0\qquad\text{for all $\vphi\in T^\perp(\pa F) \setminus \{0\}$}.
\eeq

In \cite[Theorem~1.1]{AFM} (see also \cite{JuPi} for the case of Neumann boundary conditions) it is shown that strictly stable critical sets are in fact \emph{isolated local minimizers} of the functional $J$ with respect to small $L^1$-perturbations.  
The main purpose of this paper is to show that the latter (static) stability property extends to the evolutionary case.  In Theorems~\ref{main thm 1} and \ref{main thm 2} we show that any strictly stable critical set is {\em asymptotically stable} for both \eqref{MSintro} and \eqref{SDintro}.  More precisely, we have:

\vspace{5pt}
\noindent{\bf Main Result.} {\em Let $F\subset \T^3$ be a smooth set satisfying \eqref{critintro} and \eqref{J2>0intro} (with $\gamma=0$ in the case of the surface diffusion flow).  If  $E_0$ is sufficiently close to $F$, then both the periodic modified Mullins-Sekerka flow and the periodic surface diffusion flow  starting from $E_0$  are defined for all times and converge to a translate of $F$ exponentially fast.} 
\vspace{5pt}

\noindent For the proper notion of closeness to $F$  and of exponential convergence we refer to the precise statements of the aforementioned theorems. 

Let us now comment on the class of initial data to which our main result can be applied. In the  three-dimensional case  and for the area functional ($\gamma=0$) the stable periodic sets are classified (see for instance \cite{Ros07}): they are  {\em lamellae} or {\em balls} or 
{\em cylinders} or {\em triply periodic structures} such as {\em gyroids}. It is rather easy to see that the first three configurations are in fact  strictly stable (with respect to volume preserving variations), while the strict stability of triply periodic sets has been established in some cases (see for instance in \cite{Gr, GrWo, Ross}). Due to our results, all such structures are exponentially stable for the periodic versions of \eqref{MSintro} and \eqref{SDintro}.

 As for the case $\gamma>0$ a complete classification of the stable periodic structures is still missing. However, it has been shown that lamellar configurations are strictly stable if the number of interfaces is larger than a minimum value $k(\gamma)$, where $k(\gamma)\to +\infty$ as $\gamma\to\infty$ (see \cite{AFM, ChSte}). Moreover, again by the results of \cite{AFM} one can show that if $F$ is  any periodic set that is strictly stable for the area functional, then for all $\gamma>0$ sufficiently small it is possible to find sets $F_\gamma$ that are strictly stable for \eqref{Jintro} (with the corresponding $\gamma$) in such a way that $F_\gamma\to F$ smoothly as $\gamma\to 0^+$. If instead we fix the value of $\gamma$ and $F$ is as before,  then we may find  sets $E$ that are stable for the the functional $J$ and closely resemble a rescaled version of $F$. More precisely, the following  has been shown in \cite{Cristoferi}: Let  $F\subset\T^3$ be strictly stable for the area functional,  and for any $k\in \N$ denote by $F_k$ the $1/k$-periodic set $\frac{F}k$. Then, for every $\e>0$ there exists $\bar k=\bar k(\gamma, \e)\in \N$ such that for all $k\geq \bar k$ we may find a set $E$, which is $\e$-close to $F_k$ in a $C^1$-sense and strictly stable for $J$ with respect to $1/k$-periodic variations. Moreover, the set $E$ can   be constructed in such a way that its mean curvature is  uniformly close to a constant.
Our main result clearly applies to all such sets, yielding that they are exponentially stable for the $1/k$-periodic version of the modified Mullins-Sekerka flow. 

A few comments about previous related results are in order: most of them  treat the exponential stability of $N$-dimensional spheres  both  for the Hele-Shaw  (\cite{Chen, ES98, Prockert}) and the surface diffusion flow (\cite{EMS, Wheeler}), with few exceptions in the case of the surface diffusion flow, like the infinite cylinders considered in  \cite{LeCSim, LeCSim16} and the two-dimensional triple junctions configurations studied in  \cite{GarItoKoh} (under Neumann conditions). It seems also that no asymptotic stability results for the modified Mullins-Sekerka flow were known before. Moreover, all the previous works deal with specific examples, but to the best of our knowledge no general  {\em linear versus nonlinear stability principle} has been established for \eqref{MSintro} and \eqref{SDintro} prior to  our main result. 

Most of the aforementioned papers use semigroups techniques combined with an ad hoc  center manifold analysis in order to deal with the translation invariance. Our approach instead is completely different, more variational in nature, and based on the  derivation of suitable energy identities. In this respect, our method is closer in spirit to  that  of  \cite{Chen} and  \cite{Wheeler},  where  energy identities are the key tool to establish the desired exponential stability.

Although many technical details in the proofs of our main Theorems~\ref{main thm 1} and \ref{main thm 2} are different, the underlying general argument and strategy  is the same. We overview it for the convenience of the reader.  The  starting crucial observation is that the following  energy identity holds along the flow $(E_t)_{t\in (0, T^*)}$ (see Lemmas~\ref{calculations} and \ref{calculationsbis}): Setting $\mathcal{E}(E_t):=-\frac{d}{dt} J(E_t)$, we have
\beq\label{identita1}
-\frac{d^2}{dt^2} J(E_t)=\frac{d}{dt}\mathcal{E}(E_t)=-2\pa^2 J(E_t)[V_t]+ R(E_t)\,,
\eeq
 where $\pa^2J$ is the second variation quadratic form introduced in \eqref{J2intro}, $V_t$ is the normal velocity of the moving boundary and $R(E_t)$ is  a remainder whose explicit expression depends on whether $(E_t)_t$ solves \eqref{MSintro} or \eqref{SDintro}.  Next we implement a stopping time argument; namely, we consider the maximal time $\bar t$ such that 
 \beq\label{stopping}
 \dist_{C^1}(E_t, F)< \e_0 \quad\text{and}\quad \mathcal{E}(E_t)<2\de_0 \qquad\text{for all $t\in (0, \bar t)$,}
 \eeq
where $\dist_{C^1}(E_t, F)$ stands for a suitable $C^1$-distance of $E_t$ from the stable critical set $F$ and $\e_0, \de_0$ are (small) positive constants to be chosen. Clearly, by choosing the initial set $E_0$ so close to $F$ that
\beq\label{initialintro}
 \dist_{C^1}(E_0, F)< \e_0 \quad\text{and}\quad \mathcal{E}(E_0)\leq \de_0 
 \eeq
 we can ensure that 
$\bar t>0$. The purpose is to show that $\bar t$ coincides with the maximal time of existence $T^*$. 
The argument now proceeds by contradiction, assuming that  $\bar t<T^*$ and that $ \mathcal{E}(E_{\bar t})=2\de_0$ or $\dist_{C^1}(E_{\bar t}, F)=\e_0$. Assume first that 
\beq\label{contraintro}
\mathcal{E}(E_{\bar t})=2\de_0\,.
\eeq
At this point, the idea is to exploit the strict stability assumption on $F$, and the closeness of $E_t$ to $F$ (ensured by \eqref{stopping}, with $\de_0$ smaller if needed)  to show that the quadratic form $\pa^2J (E_t)$ remains positive definite outside the space of infinitesimal translations $T(\pa E_t)$ (see \eqref{TEintro}). This observation, together with a delicate estimate  showing that $V_t$ remains bounded away from $T(\pa E_t)$, allows one to conclude that 
\beq \label{positiveintro}
\pa^2 J(E_t)[V_t]\geq \sigma \|V_t\|^2_{H^1(\pa E_t)} 
\eeq
 in $(0, \bar t)$ for a suitable constant $\sigma>0$. Next, one has to control the remainder $R(E_t)$ in \eqref{identita1}; more precisely, one shows that 
 \beq\label{remainder}
 |R(E_t)|\leq \e \|V_t\|^2_{H^1(\pa E_t)}\,,
 \eeq
where the constant $\e$ can be made arbitrarily small, provided that   $\e_0$ and $\de_0$ are chosen properly (small)  in 
\eqref{initialintro}.  The above inequality relies on delicate  boundary estimates for harmonic extensions in the case of the Mullins-Sekerka flow (see Proposition~\ref{harmonic estimates}) and on  the geometric interpolation inequality established in Lemma~\ref{nasty} in the case of the surface diffusion flow. From the technical point of view, this is  where the dimension restriction $N\leq 3$ plays a role in our argument.   Finally, one has to show that 
\beq\label{finallyintro}
 \mathcal{E}(E_{t}) \leq C\|V_t\|^2_{H^1(\pa E_t)}\,,
\eeq
with the constant $C>0$ depending only on the $C^1$-bounds on $\pa E_t$ provided by \eqref{stopping}. Collecting \eqref{identita1} and \eqref{positiveintro}--\eqref{finallyintro} yields the existence  of  $c_0>0$ such that 
$$
\frac{d}{dt}\mathcal{E}(E_t)\leq -c_0\,  \mathcal{E}(E_{t})\,,
$$
so that, by integration, 
\beq\label{expdecayintro}
\mathcal{E}(E_{t})\leq \mathcal{E}(E_{0})\mathrm{e}^{-c_0 t}\leq \de_0 \mathrm{e}^{-c_0 t}
\eeq
for $t\in [0, \bar t]$.
The above inequality contradicts \eqref{contraintro}.  Now it is not too difficult to see (using the explicit expression of $\mathcal{E}(E_{t})$) that under the $C^1$-bound of \eqref{stopping} the decay of  $\mathcal{E}(E_{t})$ obtained in 
\eqref{expdecayintro} forces $E_t$ to remain close to $F$ in a $C^1$-sense, so that assuming  $\dist_{C^1}(E_{\bar t}, F)=\e_0$ also leads to a contradiction. Thus, the stopping time $\bar t$ coincides with the maximal time and both 
\eqref{stopping} and \eqref{expdecayintro} hold for the whole lifespan of the solution. A little refinement of the estimates above allows one  also to control the H\"older-norm of the curvatures of $\pa E_t$, so that we may use the local-in-time existence theorems available for the two flows, together with a standard continuation argument, to infer that the solution exists for all times.  

Once global-in-time existence  has been established, one proceeds in the following way: A compactness argument, based on 
\eqref{stopping} and \eqref{expdecayintro}, yields the existence of a sequence $t_n\to \infty$ and of a set $F'$, critical for $J$,   such that $E_{t_n}\to F'$ (in a suitable sense). Since necessarily $F'$ is close to $F$ and of course $|F|=|F'|$, we may use the results from  \cite{AFM} (see also Proposition~\ref{prop:nocrit}) to conclude that $F'$ is a translate of $F$. The exponential convergence of the flow to 
$F'$ then  follows from \eqref{expdecayintro} via suitable elliptic estimates.

We conclude the introduction by remarking that although the presentation is restricted to the periodic case, our methods would equally work in the Neumann case, under the additional assumption that the evolving interfaces do not touch $\pa \Om$ or equivalently that $F\subset\!\subset\Om$, see Theorem~\ref{mainN}.  It would certainly be interesting to extend our result to the general Neumann setting and to arbitrary space dimensions. This will the subject of future investigations. We finally mention that  our methods would apply also to the 
{\em volume-preserving mean curvature flow} (see \cite{Hu}). However, for the sake of presentation we decided to treat only the more difficult  flows \eqref{MSintro} and \eqref{SDintro}.

The plan of the paper is the following: In Section~\ref{sec:1e2var} we introduce the precise definition of the energy functional \eqref{Jintro}, recall the formulas of the  first and the second variation and other related results that are useful for our analysis. 
In Section~\ref{sec:MS} we prove our main nonlinear stability result for the modified Mullins-Sekerka flow, while the corresponding result in the case of the surface diffusion flow is treated in Section~\ref{sec:SD}. Finally, in Section~\ref{sec:technicalproofs} we gather the proofs of several auxiliary and technical results used along the way.

\section{The nonlocal perimeter and its first and second variations}\label{sec:1e2var}

As already explained in the introduction the geometric evolutions  considered in this paper may be regarded as suitable gradient flows of (a non-local variant of) the perimeter functional. In this section  
 we introduce  such  a non-local energy and recall the first and second variation formulas, that were derived in \cite{ChSte} (see also \cite{AFM, JuPi, Mur0}).
 
 To this end,  we start by recalling that the (unit) flat torus $\T^3$ is  the quotient of  $\R^3$  with respect to 
 the equivalence relation $x\sim y \iff x-y\in \Z^3$.
The functional spaces $W^{k,p}(\T^3)$, $k\in \N$, $p\geq 1$, can be identified with the subspace of $W^{k,p}_{loc}(\R^3)$ of functions  that are one-periodic with respect to all coordinate directions. Similarly,   $C^{k,\alpha}(\T^3)$, $\alpha\in (0,1)$ may be identified with  the space of  one-periodic functions in  $C^{k,\alpha}(\R^3)$. 

A set $E\subset\T^3$ will be called of class  $W^{k,p}$, $C^k$,  or smooth if its one-periodic extension to $\R^3$ is of class $C^{k,\alpha}$,  $W^{k,p}$, or smooth. In the following we will (often) identify $E$ with such a periodic extension. 
Finally,  by saying that $E_n\to E$ in $W^{k,p}$ (or $C^{k,\alpha}$) we mean that there exists a sequence $(\Psi_n)$ of smooth diffeomorphisms from $\T^3$ to $\T^3$ such that $\Psi_n\to Id$ in $W^{k,p}$ (or $C^{k.\alpha}$) and $E_n=\Psi_n(E)$ for all $n$ sufficiently large.
When $E$ is sufficiently smooth this is equivalent  to saying that for every $\e>0$, there exists $\bar n$ such that  
\begin{multline*}
|E\Delta E_n|\leq \e\quad\text{and}\quad \pa E_n=\{x+\psi_n(x)\nu_E(x): x\in \pa E\}\,,\\
\text{ with }\|\psi_n\|_{W^{k,p}(\pa E)}\leq \e
\text{ (or  $\|\psi_n\|_{C^{k,\alpha}(\pa E)}\leq \e$) }
\end{multline*}
for all $n\geq \bar n$. Here and in the following we have used the notation $\nu_E$ to denote the outer unit normal to $E$. 
  
Given a smooth set $E\subset\T^3$, we say that a tubular neighborhood of $\pa E$ is {\em regular}, if both the {\em signed distance} function $d_E$ from the set $E$ and the orthogonal projection onto $\pa E$  are smooth functions in $U$.   Recall that 
\beq\label{etaisyys}
d_E(x) := \begin{cases}
 \text{dist}(x,\partial E) \,\, \text{if} \,\, x \not\in E, \\
  -\text{dist}(x,\partial E) \,\, \text{if} \,\, x \in E.
\end{cases}
\eeq
  
  In this periodic setting, the (relative) perimeter of a set $E\subset\T^3$ is defined as
 $$
 P_{\T^3}(E):=\sup\biggl\{\int_E \Div \vphi\, dz:\, \vphi\in C^{1}(\T^3; \R^3)\,, \|\vphi\|_\infty \leq 1 \biggr\}\,.
 $$
 Let $\gamma\geq 0$ be fixed  and for every $E\subset\T^3$ set
\beq\label{J}
J(E):= \P(E)+ \gamma \int_{\T^3} |Dv_E|^2\, dx\,,
\eeq
where  $v_E$ is the periodic  solution of 
\beq\label{eqvE}
\begin{cases}
-\Delta v_E = u_E - m,\vspace{5pt}\\
\displaystyle \int_{\T^3}v_E\, dx=0.
\end{cases}
\eeq
Here $u_E = \chi_E - \chi_{\T^3 \setminus E}$ and $m =2|E| -1$.
It is useful to recall that $v_E$ can be represented as
\beq\label{vEper}
v_E(x):=\int_{\T^3}G_{\T^3}(x, y)u_E(y)\, dy\,,
\eeq
where  $G_{\T^3}$ is the Laplacian's Green function in the torus; that is, 
for $x\in \T^3$,  $G_{\T^3}(x, \cdot)$ is the unique solution of 
$$
\begin{cases}
-\Delta_{y}G_{\T^3} (x, \cdot)=\de_{x}-1 & \text{in $\T^3$,}\\
\int_{\T^3} G_{\T^3}(x, y)\, dy=0\,.
\end{cases}
$$
We stress that the relevant particular case $\gamma=0$ (corresponding to the standard perimeter) is always included in all the discussion below. 

Throughout the paper we will make repeated use of the following notation: For any one-parameter family of functions $(g_t)_t\in (0,T)$ the symbol $\dot{g}_t$ will denote the partial derivative with respect to $s$ of the map $s\mapsto g_{t+s}$ evaluated  at $s=0$; that is, 
$$
\dot{g}_t:= \frac{\partial }{\partial s}g_{t+s}\Bigl|_{s=0}\,.
$$ 

\begin{definition}\label{def:admissibleX}
 Let $E\subset\T^{N}$ be a smooth set. 
 \begin{itemize}
 \item[(i)] We say that a one-parameter family $(\Phi_t)_{t\in I}$ of diffeomorphisms from $\T^3$ to $\T^3$, with  $I$ a real interval containing $0$, is {\em admissible}  if the map $(x, t)\mapsto\Psi_t(x)$ belongs to $C^{\infty}(\T^3\times I; \T^3)$ and  
$$
|\Phi_t(E)|=|E| \qquad\text{for all $t\in I$.}
$$
\item[(ii)]
Denote by $X_t$ the velocity field at time $t$, that is,
$$
X_t:=\dot{\Phi}_t\circ\Phi_t^{-1}
$$
and set for simplicity $X:=X_0$.   If the family $(\Phi_t)_{t\in I}$ is admissible and $X_t$ is independent of $t$, i.e., $X_t=X$, then we say that 
$(\Phi_t)_{t\in I}$ is an {\em admissible flow}.
\end{itemize}
\end{definition}
We recall that given a vector $X$, its tangential part on some smooth $(N-1)$-manifold $\mathcal{M}$  is defined as $X_\tau:=X-(X\cdot  \nu) \nu$, with $\nu$ being a unit normal vector to $\mathcal{M}$. In particular, we will denote by $D_\tau$ the tangential gradient operator given by $D_\tau\vphi:=(D\vphi)_\tau$.  Finally $\Div_\tau X$ will stand for the {\em tangential divergence} of $X$ on $\mathcal{M}$ defined as $\Div_\tau X:=\Div X-\partial_\nu X\cdot \nu$. 

\begin{theorem}[\cite{AFM, ChSte}]\label{th:12var}
 Let $E$, $(\Phi_t)_{t\in I}$, $X_t$ be as in  Definition~\ref{def:admissibleX}-(i), and set
 $$
 \dot v_E:=\frac{\pa }{\pa t}v_{\Phi_t(E)}\Bigl|_{t=0}\,,
 $$
  and $v_{\Phi_t(E)}$ is the  potential defined in \eqref{vEper}, with $E$ replaced by $\Phi_t(E)$.
 Then, 
 \beq\label{dot v}
 \dot v_E=2\int_{\pa E}G_{\T^3}(\cdot, y)X(y)\cdot \nu_E(y)\, d\Ha^{2}
 \eeq
 and
\beq\label{eq:J'}
\frac{d}{dt}J(\Phi_t(E))_{\bigl|_{t=0}}=\int_{\pa E}(H_{\pa E}+4\gamma v_E)X\cdot \nu_E\, d\Ha^{2}\,,
\eeq
where $\nu_E$ denotes the outer unit normal to $\pa E$, $H_{\pa E}$ stands for the sum of its principal curvatures, and we wrote $X$ instead of $X_0$.  
If in addition $(\Phi_t)_{t\in I}$ is an admissible flow according to Definition~\ref{def:admissibleX}-(ii), then 
\begin{align}
&\frac{d^2 }{dt^2}J(\Phi_t(E))_{\bigl|t=0} = \int_{\pa E}\Bigl(|D_\tau(X\cdot  \nu_E)|^2-|B_{\pa E}|^2(X\cdot  \nu_E)^2\Bigr)\, d\Ha^{2}\nonumber\\
& \quad+8\gamma \int_{\pa E}  \int_{\pa E}G_{\T^3}(x,y)(X\cdot  \nu_E)(x)(X\cdot  \nu_E)(y)\,d\Ha^{2}(x)\,d\Ha^{2}(y)\nonumber\\
&\quad+4\gamma\int_{\pa E}\pa_{\nu_E} v_E (X\cdot  \nu_E)^2\,d\Ha^{2}+R\,, \label{eq:J''-per}
\end{align}
where the remainder $R$ is defined as
\begin{multline}\label{Erre}
R:=-\int_{\pa E}(4\gamma v_E+H_{\pa E})\Div_{\tau}\bigl(X_\tau(X\cdot \nu_E)\bigr)\, d\Ha^{2}
\\+\int_{\pa E}(4\gamma v_E+H_{\pa E})(\Div X)(X\cdot \nu_E)\, d\Ha^{2}\,.
\end{multline}
In the above formulas $B_{\pa E}$ denotes the second fundamental form of $\pa E$ so that  the square $|B_{\pa E}|^2$  of its  Euclidean norm coincides with
  the sum of the squares of the principal curvatures. 
\end{theorem}
Recall now that if $\Phi_t$ is admissible, then $|\Phi_t(E)|=|E|$ for all $t\in [0,1]$ and thus
$$
0=\frac{d}{dt}|\Phi_t(E)|_{\bigl|_{t=0}}=\int_E \frac{d}{dt}J\Phi_{\bigl|_{t=0}}=\int_E\Div X\, dx=\int_{\pa E} X\cdot \nu_E\, d\Ha^{2}\,,
$$
that is, the normal component $X\cdot  \nu_E$ has zero average on $\pa E$. Then \eqref{eq:J'} together with a simple approximation argument (see \cite[Corollary~3.4]{AFM}) implies that 
$$
\frac{d}{dt}J(\Phi_t(E))_{\bigl|_{t=0}}=0\qquad\text{for all admissible } \Phi_t
$$
if and only if
$$
\int_{\pa E}(H_{\pa E}+4\gamma v_E)\vphi\, d\Ha^{2}=0\quad\text{for all $\vphi\in 
 C^\infty(\pa E)$ s.t. $\int_{\pa E}\vphi\, d\Ha^{2}=0$.}
$$
This motivates the following definition.
\begin{definition}[Critical sets]\label{def:criticality}
A smooth subset $F\subset\T^3$ is said to be \emph{critical} for the functional $J$ if  there exists a constant $\lambda\in \R$ such that 
$$
H_{\pa F}+4\gamma v_{F}=\lambda\qquad \text{on $\pa F$.}
$$
\end{definition}
It is now easy to see that for critical sets the remainder \eqref{Erre} vanishes so that the second variation depends (quadratically) only on $X\cdot  \nu_F$. 
Denoting 
$$
\widetilde H(\pa F):=\biggl\{\vphi\in H^1(\pa F):\, \int_{\pa F}\vphi\, d\Ha^2=0\biggr\},
$$
  we are led to consider the quadratic form $\partial^2 J(F):\widetilde H(\pa F)\to \R$ defined as 
\beq\label{J2}
\begin{split}
\partial^2 J(F)[\vphi] := &\int_{\partial F} |D_\tau \vphi|^2\, d \Ha^{2} - \int_{\partial F} |B_{\pa F}|^2  \vphi^2  \, d \Ha^{2} \\
&+ 8\gamma   \int_{\partial F}\int_{\partial F} G_{\T^3}(x,y)\vphi(x) \vphi(y) \, d \Ha^{2}(x)d \Ha^{2}(y) \\
&+ 4\gamma  \int_{\partial F} \pa_{\nu_F} v_F\, \vphi^2  \, d \Ha^{2}\,,
\end{split}
\eeq
so that if $F$ is critical, then 
$$
\frac{d^2 }{dt^2}J(\Phi_t(F))_{\bigl|t=0}=\partial^2 J(F)[X\cdot\nu_F],
$$
thanks to \eqref{eq:J''-per}. In order to give the proper notion of stability  we have to take into account that 
the functional $J$ is  invariant under translations of sets. Thus, if one consider the (admissible) flow $\Phi(t, x)=x+t\,\eta$, $\eta\in \R^3$, then $\Phi_t(F)=F+t\eta$
and  $J(\Phi_t(F))=J(F)$ for all $t$. Therefore, 
$$
0=\frac{d^2}{dt^2}J(\Phi_t(F))_{\bigl|_{t=0}}=\pa^2J(F)[\eta\cdot \nu_F]\qquad \text{for all $\eta\in \R^3$.}
$$ 
We conclude that the quadratic form $\pa^2J(F)$ always vanishes on the finite dimensional subspace  $T(\pa F)\subset\widetilde H(\pa F)$ defined as
$$
{T}(\pa F):=\bigl\{\,\eta\cdot  \nu_F:\, \eta\in \R^3\bigr\}\,.
$$
The above observation motivates the following definition.
\begin{definition}\label{def:stability+}
Let $F\subset\T^3$ be a smooth critical set, according to Definition~\ref{def:criticality}. We say that $F$ is \emph{strictly stable} if 
$$
\partial^2 J(F)[\vphi]>0\qquad\text{for all }\vphi\in T^\perp(\pa F)\setminus\{0\}.
$$
\end{definition}
Let $F$ be a smooth critical set. Observe that we may choose an orthogonal base $\{\tilde e_1, \tilde e_2, \tilde e_3\}$ of $\R^3$ such that 
  the functions $\tilde e_i \cdot \nu_F$,  $i=1,2,3$,  are orthogonal in $L^2(\partial F)$ (see \cite[Section 3]{AFM}).   Then we set
\beq\label{projektio}
\Pi_F := \text{span} \{ \tilde e_i \, : \, i \in I_F\}, 
\eeq
where 
\beq\label{indeksit}
I_F:= \{i \, : \,  \tilde e_i \cdot \nu_F \,\,  \text{ is not identically zero} \}. 
\eeq

\begin{remark}\label{rm:potential}
Setting for $\vphi\in \Ht(\pa E)$ 
\[
v_\vphi(x):=\int_{\pa E} G_{\T^3}(x, y)\vphi(y)\, d\Ha^{2}(y)
\]
and $\mu_\vphi:=\vphi\, \Ha^{2}\rtangle \pa E$,  
it follows from the properties of the Green's function (see  \cite[Chapter 18]{Landkof})  that $v_\vphi$ satisfies 
$-\Delta v_{\vphi}=\mu_\vphi$ in $\T^3$
or, equivalently,
\beq\label{wvphi}
\int_{\T^3} D v_\vphi\cdot D \psi dx=\int_{\pa E}\vphi\,\psi\, d\Ha^{2} \qquad\text{for all $\psi\in H^1(\T^3)$.}
\eeq
Therefore, 
$$
\int_{\pa E}  \int_{\pa E}G_{\T^3}(x,y)\vphi(x)\vphi(y)\,d\Ha^{2}(x)\,d\Ha^{2}(y)=
\int_{\pa E}\vphi\, v_{\vphi}\, d\Ha^{2}=\int_{\T^3}|D v_{\vphi}|^2\, dx\,,
$$
where the last equality follows from \eqref{wvphi}.
\end{remark}

In \cite[Theorem~1.1]{AFM} (see also \cite{JuPi} for the case of Neumann boundary conditions) it is shown that strictly stable critical sets are in fact \emph{isolated local minimizers} of the functional $J$ with respect to small $L^1$-perturbations.  
 It is the main purpose of this paper to show that the latter (static) stability property extends to the evolutionary case, by proving that in fact critical  configurations with positive definite second variation are asymptotically stable for suitable gradient flows of the functional $J$.
 
We conclude this section by stating two facts that will be used throughout.

The first lemma states that when a set is sufficiently close to a strictly stable critical point then the quadratic form associated with the second variation remains positive. More precisely, we have:
\begin{lemma} \label{from AFM}
Fix $p>2$ and let $F$ be a smooth strictly stable critical set in the sense of Definition~\ref{def:stability+}. Then, for every $\e\in (0,1]$ there exist  $\sigma_\e>0$ and  $\delta_1>0$  such that  
\begin{equation}\label{from AFM1}
\partial^2 J(E)[\vphi]\geq \sigma_\e\| \vphi\|_{H^1(\pa E)}^2
\end{equation}
 for all $\vphi\in \widetilde H(\pa E)$ satisfying
 $$
\min_{\eta\in \Pi_F}\|\vphi-\eta \cdot\nu_E\|_{L^2(\pa E)}\geq \e\|\vphi\|_{L^2(\pa E)}, 
$$
provided that $E\subset\T^3$ is $\de_1$-close to $F$ in a $W^{2,p}$-sense, that is 
$$
\pa E=\{x+\psi(x) \nu_F(x):\, x\in \pa F \text{ for some smooth $\psi$ with $\|\psi\|_{W^{2,p}(\pa F)}\leq \de_1$}\}.  
$$
\end{lemma}
The proof of the above lemma is  given in  Section~\ref{sec:technicalproofs}.

The final result of this section states the crucial observation that in the vicinity of a given strictly stable critical set there are no other critical sets.

\begin{proposition}\label{prop:nocrit}
Let $p$ and $F$ be as in Lemma~\ref{from AFM}. Then there exists $\de_2>0$ such that if $F'\subset\T^3$ is a smooth critical set in the sense of Definition~\ref{def:criticality}, $|F'|=|F|$, $|F\Delta F'|\leq\delta_2$ and
$$
\pa F'=\{x+\psi(x) \nu_F(x):\, x\in \pa F \text{ for some smooth $\psi$ with $\|\psi\|_{W^{2,p}(\pa F)}\leq \de_2$}\},  
$$
then $F'=F+\sigma$ for some $\sigma\in \R^3$. 
\end{proposition}
\begin{proof}
This fact is essentially proven in \cite[Proof of Theorem~3.9]{AFM}. There, it is shown that for every $p>2$ there exists $\de_2>0$ with the following property: 
if $F'\subset\T^3$ is a smooth  set with $|F'|=|F|$, $|F\Delta F'|\leq\delta_2$ and 
$$
\pa F'=\{x+\psi(x) \nu_F(x):\, x\in \pa F \text{ for some smooth $\psi$ with $\|\psi\|_{W^{2,p}(\pa F)}\leq \de_2$}\},  
$$
then we may find a small vector $\sigma\in T^3$  and an admissible flow $\Phi_t$ such that $\Phi_0(F)=(F)$, $\Phi_1(F)=F'+\sigma$ and 
$$
\frac{d^2 }{dt^2}J(\Phi_t(F))_{\bigl|t=s} \geq c |E \Delta (F'+\sigma)|^2
$$
for all $s\in [0,1]$, where $c$ is a positive constant independent of $F'$. Assume that $F'$ is a smooth critical set which is not translate of $F$. 
Then  $\frac{d }{dt}J(\Phi_t(F))_{\bigl|t=0}=0$ and from the above formula we have that  
$\frac{d}{dt}J(\Phi_t(F))_{\bigl|t=1}>0$. Therefore  $F'+\sigma$ and, in turn $F'$, is not critical. 
\end{proof}

\section{Nonlinear stability  for the modified Mullins-Sekerka  flow}\label{sec:MS}

In this section we consider the modified Mullins-Sekerka  flow.  In order to speak about  classical solutions, we need to define first  the notion of a smooth flow. 
\begin{definition}[Smooth flows of sets]\label{def:smoothflow}
We say that a one-parameter family of  sets $(E_t)_{t\in (0, T)}$ is a {\em smooth flow} on the interval $(0, T)$ 
if there exists a smooth {\em reference set} $F\subset\T^3$ and a map $\Psi\in C^{\infty}(\T^3\times (0, T); \T^3)$ such that $\Psi_t:=\Psi(\cdot, t)$ is a smooth diffeomorphism from $\T^3$ into $\T^3$  and  $E_t=\Psi_t(F)$ for all $t\in [0, T)$.
\end{definition}

We will make use of the following notation: Given a (smooth) set $E\subset\T^3$, we denote by $w_E$ the unique solution in $H^1(\T^3)$ to the following problem
\beq\label{WE}
\begin{cases}
\Delta w_E=0 & \text{in }\T^3\setminus \pa E\\
w_E= H_{\pa E} + 4\gamma v_{E}  & \text{on } \, \partial E,
\end{cases}
\eeq
 where $v_E$ is the potential introduced in \eqref{eqvE}.  Moreover, we denote by $w^+_E$ and $w^-_E$ the restrictions $w_E|_{\T^3\setminus E}$ and ${w_E}{|_{E}}$, respectively. Finally, denoting as usual by $\nu_E$ the outer unit normal to $E$, we set
 $$
 [\pa_{\nu_E} w_E]:=\pa_{\nu_E}w^+_E-\pa_{\nu_E}w^-_E=-(\pa_{\nu_{E^c}}w^+_E+\pa_{\nu_E}w^-_E)\,.
 $$
 In the following, given $\alpha\in (0,1)$ and $k, m\in \N$ we denote
 $$
 h^{k,\alpha}(\R^m):=\{f\in C^{k,\alpha}(\R^m):\, \exists \{f_n\}\subset C^{\infty}(\R^m)\text{ s.t. }
 f_n\to f\text{ locally in }C^{k,\alpha}(\R^m)\}\,.
 $$
 The space $h^{k,\alpha}(M)$, when $M\subset \R^m$ is a smooth manifold can be then defined by means of local charts. 
 In turn,  we will say that a set $F\subset\T^3$ is of class $h^{k,\alpha}$, $\alpha\in (0,1)$, if for each point $x\in\pa F$ there exists a a neighborhood $V$ of $x$,  a function $f\in h^{k,\alpha}(\R^{2})$, and a suitable coordinate system such that
 $F\cap V=\{(x', x_N)\in V:\, x_N\leq f(x')\}$. 
 \begin{definition}[Nonlocal Mullins-Sekerka flows]\label{def:OKsol}
 Let $E_0\subset\T^3$ be of class $h^{2,\alpha}$ for some $\alpha\in (0,1)$. We say that  the one-parameter family $(E_t)_{t\in (0, T)}$ is a  {\em classical solution} to the modified Mullins-Sekerka  flow on the interval $(0, T)$ with initial datum $E_0$ if it is a smooth flow in the sense of  Definition~\ref{def:smoothflow}, $E_t\to E_0$ in $C^{2,\alpha}$ as $t\to 0^+$, and  the following evolution law holds:
\begin{equation}\label{MSnl}
V_t= [\partial_{\nu_t} w_{t}] \quad\text{on } \, \partial E_t\text{ for all }t\in (0, T)\,,
\end{equation}
where $V_t$ stands for the outer normal velocity of the moving boundary $\pa E_t$. Here  we used the simplified notation 
$\partial_{\nu_t} w_{t}$ in place  of $\partial_{\nu_{E_t}} w_{E_t}$.  
\end{definition}
As explained in the introduction the modified Mullins-Sekerka flow is volume preserving. This can be easily checked by the following computation (using also the notation introduced in Definition~\ref{def:OKsol}):
$$
\frac{d}{dt}|E_t|=\int_{\pa E_t}V_t\, d\Ha^{2}=\int_{\pa E_t}[\pa_{\nu_t} w_t]\, d\Ha^{2}=0\,, 
$$
where the last equality follows from the Divergence Theorem and the fact that $w_t$ is harmonic in $\T^3\setminus \pa E_t$. 

We use the following notation: Given a smooth set $F\subset\T^3$ and a  regular tubular neighborhood $U$ of $\pa F$,  we denote by $\mathfrak{C}^1_M(F, U)$, $M>0$,
the class of all smooth sets $E\subset F\cup U$ such that  
\beq\label{front}
\pa E=\{x+\psi_E(x)\nu_{F}(x):\, x\in \pa F \}\,,
\eeq
for some $\psi_E\in C^\infty(\pa F)$, with $\|\psi\|_{C^1(\pa F)}\leq M$. For $\alpha\in (0,1)$ and $k\in\N$ we also let $\mathfrak{h}^{k,\alpha}_M(F, U)$ be the collection of sets $E\in \mathfrak{C}^1_M(F, U)$ such that $\|\psi\|_{h^{k,\alpha}(\pa F)}\leq M$.
We are now ready to state a local-in-time existence and uniqueness result proved in \cite{EscNis}. \footnote{In fact \cite{EscNis} deals with the evolution in the whole space $\R^N$, but it is clear that the same arguments go through in the periodic case.}

\begin{theorem}[Local-in-time existence and uniqueness, \cite{EscNis}]\label{th:EscNis}
 Let $F_0\subset\T^3$ be a smooth set and   $U$ a regular tubular neighborhood of $\pa F_0$.  Then, for every  $M>0$ and $\alpha\in (0,1)$ there exists $T>0$ with the following property:  For every $E_0\in \mathfrak{h}^{2,\alpha}_M(F_0, U)$
there exists a unique classical solution  to the modified Mullins-Sekerka  flow in $(0, T)$  with initial datum $E_0$.
 \end{theorem}

Our purpose is to show that for special initial data the flow exists for all time and then to study its long-time behavior. 

The main result is the following.
\begin{theorem}[Main result] \label{main thm 1}
Let $F\subset\T^3$ be a strictly stable critical set according to Definition~\ref{def:stability+} and let $U$ be a regular tubular neighborhood of $\pa F$. Then, for every  $M>0$ and $\alpha\in (0,1)$ there exists $\de_0>0$ with the following property: Let $E_0\in \mathfrak{h}^{2,\alpha}_M(F, U)$ be  such that 
$$
|E_0|=|F|\,, \qquad |E_0\Delta F|\leq \de_0\,, \qquad \text{and}\qquad \int_{\T^3}|Dw_{E_0}|^2\, dx\leq\de_0\,.
$$
 Then,   the unique classical solution $(E_t)_t$ to  the Mullins-Sekerka flow with initial datum $E_0$ is  defined  for all $t>0$. Moreover,  $E_t\to F+\sigma$ in 
 $W^{5/2,2}$ exponentially fast as $t\to +\infty$, for some $\sigma\in \R^3$. More precisely, there exist $\eta$, $c_F>0$  such that for all $t>0$, writing 
\[
\pa E_t=\{x+\psi_{\sigma, t}(x)\nu_{F+\sigma}(x):\, x\in \pa F+\sigma \}\,,
\]
we have 
$$
\|\psi_{\sigma, t}\|_{W^{5/2,2}(\pa F+\sigma)}\leq \eta\mathrm{e}^{-c_Ft}\,.
$$
Both $|\sigma|$ and $\eta$ vanish as $\de_0\to 0^+$.  
\end{theorem}
The proof of the result is postponed until the end of this section. It will be achieved through several auxiliary results, that we state in the following and whose proofs can be found in the final section.
\begin{lemma}[Energy identities] \label{calculations} Let $(E_t)_{t\in (0, T)}$ be a smooth flow satisfying \eqref{MSnl}. The following energy idienties hold:
\begin{equation}
\label{der of J}
\frac{d}{dt} J(E_t) = -  \int_{\T^3} |D w_t|^2\, dx\,,
\end{equation}
and
\begin{equation}
\label{der of dw}
\frac{d}{dt} \left(\frac{1}{2} \int_{\T^3} |D w_t|^2\, dx \right) = -\pa^2J(E_t)\left[[\pa_{\nu_t}w_t\vphantom{^{^4}}]\right]
+ \frac{1}{2}\int_{\partial E_t} (\partial_{\nu_t} w^+_t+ \partial_{\nu_t} w_t^-) [\partial_{\nu_t} w_t]^2   \, d \Ha^{2} \,,
\end{equation}
where $\pa^2J(E_t)$ is the  quadratic form defined in \eqref{J2} (with $E_t$ in place of $E$) and, as usual, the subscript $t$ stands for  ${E_t}$.
\end{lemma}
The proof of the lemma is given in the final section. Note that if $E_t$ is not critical then  $\frac{d^2}{dt^2} J(E_t)$ is not equal to  
the second variation of  $J(E_t)$ evaluated at $[\pa_{\nu_t}w_t\vphantom{^{^4}}]$. However, quite surprisingly
 the formulas above show that  the leading order term of $\frac{d^2}{dt^2} J(E_t)$  is indeed twice 
the quadratic form $\pa^2J(E_t)$ at $[\pa_{\nu_t}w_t\vphantom{^{^4}}]$.  The same holds for the  surface diffusion flow, see \eqref{der of DH}. The next proposition provides crucial boundary estimates for harmonic functions. Some of them are perhaps well-known to the experts. However, for the convenience of the reader we provide a self-contained proof in the final section.

\begin{proposition}[Boundary estimates for harmonic functions]
\label{harmonic estimates}
Let $E \subset \T^3$ be of class $C^{1,\alpha}$,  $f \in C^\alpha(\partial E)$ (with zero average on $\partial E$) and let $u \in H^{1}(\T^3)$ be the solution of 
\[
-\Delta u = f \Ha^{2} \,\, \rtangle{\partial E}
\]  
with zero average in $\T^3$. Denote   $u^- = u \big|_{E}$ and $u^+ = u \big|_{\T^3 \setminus E}$ and assume that $u^-$ and $u^+$ are of class $C^1$ up to the boundary  $\pa E$. Then, for every $1 < p < \infty$ there exists a constant $C$,  which depends only on the $C^{1,\alpha}$
bounds on $\partial E$ and on $p$, such that: 
\begin{enumerate}
\item[(i)] 
\[
\| u \|_{L^p(\partial E)} \leq C \|f\|_{L^p(\partial E)}\,;
\]
\item[(ii)]
\[
\| \partial_{\nu_E} u^+\|_{L^2(\partial E)} +  \| \partial_{\nu_E} u^-\|_{L^2(\partial E)} \leq  C \|u \|_{H^1(\partial E)}\,;
\]  
\item[(iii)] 
\[
\| \partial_{\nu_E} u^+\|_{L^p(\partial E)} +  \| \partial_{\nu_E} u^-\|_{L^p(\partial E)} \leq  C \|f\|_{L^p(\partial E)}\,.
\]
 \item[(iv)]
$$
\|u\|_{C^{0, \beta}(\pa E)}\leq C \|f\|_{L^{p}(\pa E)}
$$ 
for all $\beta\in (0, \frac{p-2}{p})$, with  $C$ depending also on $\beta$.
\item[(v)] 
Moreover, if $f \in H^1(\partial E)$, then for every $2\leq p< +\infty$ there exists a constant $C$, which depends only on the $C^{1,\alpha}$
bounds on $\partial E$ and on $p$, such that 
\[
\| f\|_{L^p(\partial E)} \leq C \|f\|_{H^1(\partial E)}^{\frac{p-1}{p}}\,  \|u \|_{L^2(\partial E)}^{\frac1p}.
\]

\end{enumerate}
\end{proposition}

We will need also the following:
\begin{lemma}[Compactness of sets]\label{w52conv}
Let $F\subset\T^3$ be a smooth set and denote by $U$ a fixed regular tubular neighborhood of $\pa F$. Let $\{E_n\}_n\subset \mathfrak{C}^1_M(F, U) $ be a sequence of sets  such that 
$$
\sup_n\int_{\T^3}|Dw_{E_n}|^2\, dx<+\infty\,.
$$
Then there exists $F'\in \mathfrak{C}^1_M(F, U)$ of class  $W^{\frac{5}{2},2}$ such that, up to a  (non relabeled) subsequence, $E_n\to F'$ in $W^{2,p}$ for all $1\leq p<4$. Moreover, if
$$
\int_{\T^3}|Dw_{E_n}|^2\, dx\to 0\,,
$$
then $F'$ is critical in the sense of Definition~\ref{def:criticality} and the convergence holds in $W^{\frac52,2}$.
\end{lemma}

We give now the proof of Theorem \ref{main thm 1}.

\begin{proof}[\textbf{Proof of Theorem \ref{main thm 1}.}]
 Throughout the proof   $C$ will denote a constant depending only on the $C^{1,\alpha}$-bounds on the boundary of the set. The value of $C$  may change from line to line.  
 We start by the trivial observation that if $\{E_n\}_n\subset \mathfrak{h}^{2,\alpha}_M(F, U)$ and $|E_n\Delta F|\to 0$, then 
 $E_n\to F$ in $C^{2,\beta}$ for all $\beta\in (0,\alpha)$. For any set $E\in \mathfrak{C}^{1}_M(F, U)$  consider  
 \beq\label{D(E)0}
 D(E):=\int_{E\Delta F}\mathrm{dist\,}(x, \pa F)\, dx=\int_E d_F\, dx-\int_Fd_F\, dx,
 \eeq
where $d_F$ is the signed distance function defined in \eqref{etaisyys}. Using coarea formula the reader may check that 
 \beq\label{D(E)}
 |E\Delta F|\leq C\|\psi_E\|_{L^1(\pa F)} \leq C\|\psi_E\|_{L^2(\pa F)}\leq C\sqrt{D(E)}
 \eeq
for a constant $C$ depending only on $F$.  For every  $\e_0>0$ sufficiently small, there exists $\de_0\in (0,1)$ so small that for any set $E\in \mathfrak{C}^{1}_M(F, U)$ the following implications hold true:
 \beq\label{de01}
 E\in \mathfrak{h}^{2,\alpha}_M(F, U)\text{ and } D(E)\leq \de_0\Longrightarrow \|\psi_E\|_{C^1(\pa F)}\leq \frac{\e_0}2\,
 \eeq  
 and 
 \beq\label{de02}
  \|\psi_E\|_{C^1(\pa F)}\leq \e_0\text{ and } \int_{\T^3} |D w_{E}|^2\, dx \leq 1\Longrightarrow \|\psi_E\|_{W^{2,3}(\pa F)}\leq \omega(\e_0)\leq 1\,,
 \eeq  
 where $\omega$ is a positive non-decreasing function such that $\omega(\e_0)\to 0$ as $\e_0\to 0^+$. The last implication is true thanks to Lemma~\ref{w52conv}.
 Fix $\e_0$, $\de_0\in (0,1)$ satisfying \eqref{de01} and \eqref{de02} and choose an initial set $E_0\in \mathfrak{h}^{2,\alpha}_M(F, U)$ such that
 \beq\label{initial}
 D(E_0)\leq \delta_0\qquad\text{and}\qquad \int_{\T^3} |D w_{E_0}|^2\, dx \leq \delta_0\,.
 \eeq
 Let $(E_t)_{t\in (0, T(E_0))}$ be the unique classical solution to the modified Mullins-Sekerka flow provided 
 by Theorem~\ref{th:EscNis}. Here $T(E)\in (0,+\infty]$ stands for  the maximal time of existence of the classical  solution starting from $E$.  By the same theorem, there exists $T_0>0$ such that 
 \beq\label{dalbasso}
 T(E)\geq T_0\qquad\text{for  all $E\in \mathfrak{h}^{2,\alpha}_M(F, U)$.}
 \eeq
 We now split the rest of the proof into several steps. 
 
 \vspace{5pt}
\noindent {\bf Step 1.}{\it (Stopping-time)} Let $\bar t\leq T(E_0)$ be the maximal time such that 
\beq\label{Tprimo}
\|\psi_t\|_{C^1(\pa F)}<\e_0\quad\text{and}\quad \int_{\T^3} |D w_{t}|^2\, dx < 2\delta_0
\quad\text{for all $t\in (0, \bar t)$,}
\eeq
with $\eps_0>0$ a suitable  constant that will be chosen below. Here and in the following the subscript $t$ stands for the subscript $E_t$. 
 Note that such a maximal time is well defined in view of \eqref{de01} and \eqref{initial}. We claim that  by taking $\de_0$ smaller if needed, we have $\bar t=T(E_0)$.
 
\vspace{5pt}

\vspace{5pt}
\noindent {\bf Step 2.}{\it (Estimate of the  translational component of the flow)} We claim that there exists small $\e>0$ such that
\begin{equation} \label{not a translation}
\min_{\eta\in \Pi_F}\big\|\,[\pa_{\nu_t}w_t]- \eta\cdot\nu_t\big\|_{L^2(\pa E_t)}\geq \e\|[\pa_{\nu_t}w_t]\|_{L^2(\pa E_t)}\qquad\text{for all }t\in (0, \bar t)\,,
\end{equation}
where $\Pi_{F}$ is defined in \eqref{projektio}. To this aim, let $\eta_t\in \Pi_{F}$ be such that 
\begin{equation} \label{not a translation 2}
[\partial_{\nu_t} w_t]  = \eta_t\cdot \nu_t + g,
\end{equation} 
where $g$ is orthogonal to the subspace of $L^2(\partial E_t)$ spanned by $\tilde e_i\cdot \nu_{t}$ with $i \in I_F$ (see \eqref{indeksit}).  We argue by contradiction assuming $\|g\|_{L^2(\pa E_t)} < \eps \| [\partial_{\nu_t} w_t]   \|_{L^2(\pa E_t)}$. 
 First of all, by \eqref{eq:J'} and the translation invariance of the energy we have
 $$
 0=\frac{d}{ds}J(E_t+s\eta_t)_{\bigl|_{s=0}}=\int_{\pa E_t}(H_{t}+4\gamma v_t)\eta_t\cdot \nu_t\, d\Ha^{2}=\int_{\pa E_t}w_t(\eta_t\cdot \nu_t)\, d\Ha^{2}\,.
 $$
Thus, multiplying \eqref{not a translation 2} by $w_t-\hat w_t$, with $\hat w_t:=\medintinrigo_{\T^3}w_t\, dx$,  and integrating over $\partial E_t$,  we get
\beq\label{trans}
\begin{split}
\int_{\T^3}|Dw_t|^2 \, dx &=-\int_{\pa E_t}w_t [\pa_{\nu_t}w_t]\, d\Ha^2= -\int_{\pa E_t}(w_t-\hat w_t) [\pa_{\nu_t}w_t]\, d\Ha^2\\
&= -\int_{\pa E_t}(w_t-\hat w_t) g\, d\Ha^2\\
&\leq  \eps \| w_t - \hat w_t \|_{L^2(\partial E_t)} \|[\partial_{\nu_t} w_t]   \|_{L^2(\partial E_t)}.
\end{split}
\eeq
Note that in the second  and the third equality above we have used the fact that $[\pa_{\nu_t}w_t]$ and $\nu_t$, respectively, have zero average on $\pa E_t$.  
Let us denote the (periodic) harmonic extension of $\eta_t\cdot\nu_t$ to $\T^3$ by $f$. Since 
\[
\int_{\partial F} |\tilde e_i \cdot \nu_F|^2 \,  d\Ha^{2} >0 \qquad \text{for }\, i \in I_F
\]
from \eqref{Tprimo} it follows that if $\eps_0$ is small enough then $||\tilde e_i \cdot \nu_t||_{L^2(\partial E_t)} \geq c_0 >0$ for all $i \in I_F$. Hence
 $|\eta_t|\leq C\|[\partial_{\nu_t} w_t]   \|_{L^2(\partial E_t)}$. By \eqref{de02} we have 
\begin{equation} \label{bound for harm ext}
\|Df\|_{L^2(\T^3)} \leq C \|\eta_t\cdot \nu_t\|_{H^{1/2}(\partial E_t)} \leq C |\eta_t| \|\nu_t\|_{W^{1,3}(\pa E_t)}\leq 
C \|[\partial_{\nu_t} w_t]   \|_{L^2(\partial E_t)}\,.
\end{equation}
Note now that 
\begin{equation} \label{the equation jump}
\Delta w_t = [\partial_\nu w_t]  \Ha^{2}  \rtangle{\partial E}\qquad\text{in $\T^3$.}
\end{equation}
 We may then apply Proposition~\ref{harmonic estimates}-(i) to obtain
\begin{equation} \label{from poincare}
\| w_t - \hat w_t \|_{L^2(\partial E_t)}  \leq C\|[\partial_{\nu_t} w_t]   \|_{L^2(\partial E_t)}.
\end{equation}
Thus, combining \eqref{not a translation 2} with \eqref{trans}--\eqref{from poincare}, we infer
\[
\begin{split}
\|\eta_t\cdot\nu_t\|^2_{L^2(\pa E_t)}&=\int_{\pa E_t}[\pa_{\nu_t}w_t] (\eta_t\cdot\nu_t)\, d\Ha^2= - \int_{\T^3}  Df\cdot  Dw_t \, dx\\
&\leq \left(\int_{\T^3} |Df|^2 \,dx\right)^{1/2}  \left(\int_{\T^3}  |Dw_t|^2\, dx\right)^{1/2} \\
&\leq C  \eps^{1/2}\|[\partial_{\nu_t} w_t] \|^2_{L^2(\partial E_t)}\,. 
\end{split}
\]
If $\e$ is chosen so small that $C\e^{\frac12}+\e^2 <1$ in the last inequality, then we reach a contradiction to 
\eqref{not a translation 2} and the fact that $\|g\|_{L^2(\pa E_t)}<\e \| [\partial_{\nu_t} w_t]   \|_{L^2(\pa E_t)}$. This shows that for this choice of $\e$ condition \eqref{not a translation} holds.
Recall now that by Lemma~\ref{from AFM} and Proposition~\ref{prop:nocrit}, there exist $\sigma_\e$ and  $\de_1>0$ with the following properties: for any set $E\in \mathfrak{C}^{1}_M(F, U)$ 
\beq\label{de03}
\begin{split}
  \|\psi_E\|_{W^{2,3}(\pa F)}\leq\de_1\Longrightarrow&\,\,
 \partial^2 J(E)[\vphi]\geq \sigma_\e\| \vphi\|_{H^1(\pa E)}^2 \text{ for all $\vphi\in \widetilde H(\pa E)$}\\
& \text{ s.t. 
 $\min_{\eta\in \Pi_F}\|\vphi-\eta\cdot\nu_E\|_{L^2(\pa E)}\geq \e\|\vphi\|_{L^2(\pa E)}$}
 \end{split}
 \eeq
 and 
 \beq\label{de04}
 \text{ $F'$  critical, }|F|=|F'|\quad\text{ and}\quad \|\psi_{F'}\|_{W^{2,3}(\pa F)}\leq\de_1
 \Longrightarrow F'=F+\sigma
 \eeq
 for a suitable $\sigma\in \R^3$. By taking $\e_0$ (and $\de_0$) smaller, if needed, we may ensure that 
 \beq\label{de05}
 \omega(\e_0)\leq \de_1\,,
 \eeq
 where $\omega$ is the modulus of continuity introduced in \eqref{de02}.
 
 \vspace{5pt}
\noindent {\bf Step 3.}{\it (The stopping time $\bar t$ equals the maximal time $T(E_0)$)} Here we show that, by taking $\de_0$ smaller if needed, we have $\bar t=T(E_0)$. To this aim, assume by contradiction that $\bar t< T(E_0)$. Then,   
$$
\|\psi_{\bar t}\|_{C^1(\pa F)}=\e_0\quad\text{or}\quad \int_{\T^3} |D w_{\bar t}|^2\, dx =2\delta_0
$$
We further split into two sub-steps, according to the two alternatives above.

\vspace{5pt}
{\it Step 3-(a).}  Assume that 
\beq\label{step3a}
\int_{\T^3} |D w_{\bar t}|^2\, dx =2\delta_0
\eeq
Recall that \eqref{not a translation} holds. Thus, by \eqref{de02}, \eqref{Tprimo},  \eqref{de03}, and \eqref{de05} we have
$$
\partial^2 J(E_t)\left[[\pa_{\nu_t}w_t\vphantom{^{^4}}]\right]\geq \sigma_\e\|[\pa_{\nu_t}w_t] \|_{H^1(\pa E)}^2 \text{ for all $t\in (0, \bar t)$.}
$$

  In turn, by Lemma~\ref{calculations}   we may estimate 
\[
\begin{split}
\frac{d}{dt} \left(\frac{1}{2} \int_{\T^3} |D w_t|^2\, dx \right) \leq &-\sigma_\eps \|[\pa_{\nu_t}w_t] \|_{H^1(\pa E)}^2+ \frac{1}{2}\int_{\partial E_t} (\partial_{\nu_t} w^+_t+ \partial_{\nu_t} w_t^-) [\partial_{\nu_t} w_t]^2   \, d \Ha^{2}
\end{split}
\]
for every $t \leq \bar t$. By Proposition~\ref{harmonic estimates}-(iii) and \eqref{the equation jump}, we may estimate the last term by
\[
\begin{split}
\int_{\partial E_t} (\partial_{\nu_t} w^+_t+ \partial_{\nu_t} w_t^-) [\partial_{\nu_t} w_t]^2   \, d \Ha^{2} &\leq C \int_{\partial E_t} (|\partial_{\nu_t} w^+_t|^3 + |\partial_{\nu_t} w^-_t|^3) \, d \Ha^{2} \\
&\leq C \int_{\partial E_t}| [\partial_{\nu_t} w_t] |^3  \, d \Ha^{2}. 
\end{split}
\]
Now,  Proposition~\ref{harmonic estimates}-(v)  implies
\[
\|[\partial_{\nu_t} w_t] \|_{L^3(\partial E_t)} \leq C \|[\partial_{\nu_t} w_t]  \|_{H^1(\partial E_t)}^{2/3}  \|w_t - \hat w_t\|_{L^2(\partial E_t)}^{1/3}.
\]  
Therefore, combining the last three estimates,  we get 
\beq\label{quasi}
\begin{split}
\frac{d}{dt} \left(\frac{1}{2} \int_{\T^3} |D w_t|^2\, dx \right)
&\leq -\sigma_\eps \|[\pa_{\nu_t}w_t] \|_{H^1(\pa E_t)}^2+C\|w_t - \hat w_t\|_{L^2(\partial E_t)}  \|[\pa_{\nu_t}w_t] \|_{H^1(\pa E_t)}^2\\
&\leq -\frac{\sigma_{{\eps}}}{2} \|[\pa_{\nu_t}w_t] \|_{H^1(\pa E_t)}^2
\end{split}
\eeq
for every $t \leq \bar t$,  where the last inequality holds provided that $\delta_0$ is small enough since by \eqref{Tprimo} and by trace theorem
$$
\|w_t-\hat w_t\|^2_{L^2(\pa E_t)}\leq C\int_{\T^3}|Dw_t|^2\, dx\leq C\delta_0\,.
$$

 We use \eqref{from poincare} to conclude
\[
\begin{split}
 \int_{\T^3} |D w_t|^2\, dx &=- \int_{\partial E_t} w_t [\partial_{\nu_t} w_t] \, d \Ha^{2} = - \int_{\partial E_t} (w_t - \hat w_t) [\partial_{\nu_t} w_t] \, d \Ha^{2} \\
&\leq  \| w_t - \hat w_t \|_{L^2(\partial E_t)}  \|[\partial_{\nu_t} w_t]\|_{L^2(\partial E_t)} \\
&\leq C \|[\partial_{\nu_t} w_t]\|_{L^2(\partial E_t)}^2\,.
\end{split}
\]
Combining the above inequality with \eqref{quasi}, we finally obtain  
\[
\frac{d}{dt} \int_{\T^3} |D w_t|^2\, dx \leq -c_0  \int_{\T^3} |D w_t|^2\, dx
\]
for every $t \leq \bar t$ and for a suitable $c_0>0$.  Integrating the differential inequality  and recalling \eqref{initial}, we get
\beq\label{expfinalmente}
\int_{\T^3} |D w_{ t}|^2\, dx \leq \mathrm{e}^{-c_0 t}\int_{\T^3} |D w_{E_0}|^2\, dx \leq \delta_0  \mathrm{e}^{-c_0 t}\,,
\eeq
which for $t=\bar t$ gives a contradiction to \eqref{step3a}.

\vspace{5pt}
{\it Step 3-(b).}  Assume that 
\beq\label{step3b}
\|\psi_{\bar t}\|_{C^1(\pa F)}=\e_0\,.
\eeq
Recalling \eqref{D(E)0} and denoting by $X_t$ the velocity field of the flow (see Definition~\ref{def:admissibleX}), we may compute
\[
\begin{split}
\frac{d}{dt}D(E_t)&=\frac{d}{dt}\int_{E_t} d_F\, dx= \int_{E_t}\Div(d_F X_t)\, dx\\
&= \int_{\pa E_t}d_F(X_t\cdot\nu_t)\, d\Ha^2= \int_{\pa E_t}d_F [\pa_{\nu_t}w_t]\, d\Ha^2\\
&=- \int_{\T^3}Dh\cdot Dw_t\, dx\,,
\end{split}
\]
where $h$ denotes the harmonic extension of $d_F$ to $\T^3$. Note that
$$
\|Dh\|_{L^2(\T^3)}\leq C\|d_F\|_{C^1(\pa E_t)}\leq C\,.
$$
Thus, also by \eqref{expfinalmente}, we have
\[
\frac{d}{dt} D(E_t) \leq C  \|Dw_t\|_{L^2(\T^3)}\leq C\sqrt{\de_0}\mathrm{e}^{-\frac{c_0}2t}
\]
for all $t \leq \bar t$. By integrating over $(0, \bar t)$ and recalling \eqref{D(E)} we get
\beq\label{step33}
\|\psi_{\bar t}\|_{L^2(\pa F)}\leq C\sqrt{D(E_{\bar t})}\leq C\sqrt{D(E_0)+C\sqrt{\de_0}}\leq C\sqrt[4]{\de_0}\,,
\eeq
provided that $\de_0$ is small enough. Since by \eqref{Tprimo} and \eqref{de02} we also have uniform $W^{2,3}$-bounds on 
$\psi_{\bar t}$, by standard interpolation we infer from \eqref{step33} that $\|\psi_{\bar t}\|_{C^1(\pa F)}\leq C\de_0^{\theta}$ for a suitable $\theta\in (0,1)$. Thus if $\de_0$ is small enough we reach a contradiction to \eqref{step3b}.

The combination of Step 3-(a) (see also \eqref{expfinalmente}) and Step 3-(b) yields $\bar t=T(E_0)$ and 
\beq\label{finaldecay}
\|\psi_t\|_{C^1(\pa F)}<\e_0\text{ and }\int_{\T^3} |D w_{t}|^2\, dx \leq \mathrm{e}^{-c_0 t}\int_{\T^3} |D w_{E_0}|^2\, dx \quad\text{ for all $t\in (0, T(E_0))$.}
\eeq

\vspace{5pt}
\noindent {\bf Step 4.}{\it (Global-in-time existence)} Here we show that, by taking $\de_0$ smaller if needed, we have $T(E_0)=+\infty$, that is the classical solution exists for all times. To this aim, recall that by \eqref{quasi} and the fact that $\bar t=T(E_0)$ we have
$$
\frac{d}{dt} \left(\frac{1}{2} \int_{\T^3} |D w_t|^2\, dx \right)+\frac{\sigma_{{\eps}}}{2} \|[\pa_{\nu_t}w_t] \|_{H^1(\pa E)}^2\leq 0
$$
for all $t\in (0, T(E_0))$. Assume now by contradiction $T(E_0)<+\infty$. Integrating over $\left(T(E_0)-\frac{T_0}2,T(E_0)-\frac{T_0}4 \right)$, where $T_0$ is as in  
\eqref{dalbasso}, we obtain
\[
\begin{split}
\sigma_{{\eps}}\int_{T(E_0)-\frac{T_0}2}^{T(E_0)-\frac{T_0}4}\|[\pa_{\nu_t}w_t] \|_{H^1(\pa E_t)}^2\, dt&\leq 
\int_{\T^3} |D w_{T(E_0)-\frac{T_0}2}|^2\, dx- \int_{\T^3} |D w_{T(E_0)-\frac{T_0}4}|^2\, dx\\
&\leq \de_0\,,
\end{split}
\]
where the last inequality  follows from \eqref{finaldecay} and \eqref{initial}.
Thus, by the mean value theorem there exists $\hat t\in \left(T(E_0)-\frac{T_0}2,T(E_0)-\frac{T_0}4 \right)$ such that
 $\|[\pa_{\nu_{\hat t}}w_{\hat t}] \|_{H^1(\pa E_t)}^2\leq \frac{8\de_0}{T_0\sigma_\e}$. Since $H^1(\pa E_{\hat t})$ embeds into $L^p(\pa E_{\hat t})$ for all $p>1$, by Proposition~\ref{harmonic estimates} we in turn infer that 
 $$
 [H_{\hat t}]^2_{C^{0,\alpha}(\pa E_{\hat t})}\leq C [w_{\hat t}]^2_{C^{0,\alpha}(\pa E_{\hat t})}\leq C\frac{\de_0}{T_0\sigma_\e}\,,
 $$
 where $[\cdot ]_{C^{0,\alpha}(\pa E_{\hat t})}$ stands for the $\alpha$-H\"older seminorm on $\pa E_{\hat t}$. Thus, if we choose 
 $\de_0$ sufficiently small, the above inequality together with \eqref{Tprimo} ensures that 
 $E_{\hat t}\in \mathfrak{h}^{2,\alpha}_M(F, U)$. In turn, by \eqref{dalbasso} the   time span   of existence of the classical solution starting from  $E_{\hat t}$ is at least $T_0$, which means that $(E_t)_t$ can be continued beyond $T(E_0)$. This is clearly a contradiction.

\vspace{5pt}
\noindent {\bf Step 5.}{\it (Convergence, up to subsequences, to a translate of $F$)} Let $t_n\to +\infty$. Then by \eqref{finaldecay} the sets $E_{t_n}$ satisfy the hypotheses of Lemma~\ref{w52conv}. Thus, up to a (not relabeled) subsequence we have that there exists a critical set $F'\in \mathfrak{C}^1_M(F, U)$ such that $E_{t_n}\to F'$ in $W^{\frac52,2}$. Due to \eqref{de02} and \eqref{de05} we also have $\|\psi_{F'}\|_{W^{2,3}(\pa F)}\leq \delta_1$.  But then \eqref{de04} implies that $F'=F+\sigma$ for a suitable (small) $\sigma\in \R^3$.

\vspace{5pt}
\noindent {\bf Step 6.}{\it (Exponential convergence of the full sequence)} Consider  now the $L^2$-distance of $\pa E_t$ from 
$\pa F+\sigma$:
$$
D_\sigma(E):=\int_{E\Delta (F+\sigma)}\mathrm{dist\,}(x, \pa F+\sigma)\, dx\,.
$$
The very same calculations performed in Step 3-(b) show that 
\beq\label{step6}
\frac{d}{dt} D_\sigma(E_t) \leq C  \|Dw_t\|_{L^2(\T^3)}\leq C\sqrt{\de_0}\mathrm{e}^{-\frac{c_0}2t}
\eeq
for all $t>0$. From this inequality it is easy to deduce that $\lim_{t\to +\infty} D_\sigma(E_t)$ exists. Thus, by the previous step  
$D_\sigma(E_t)\to 0$ as $t\to +\infty$. In turn, integrating \eqref{step6} and writing $\pa E_t=\{x+\psi_{\sigma, t}(x)\nu_{F+\sigma}(x): x\in \pa F+\sigma\}$ we get
\beq\label{step61}
\|\psi_{\sigma, t}\|_{L^2(\pa F+\sigma)}^2\leq C D_\sigma(E_t)\leq\int_t^{+\infty}C\sqrt{\de_0}\mathrm{e}^{-\frac{c_0}2s}\, ds\leq 
C\sqrt{\de_0}\mathrm{e}^{-\frac{c_0}2t}\,.
\eeq
Since by the previous steps $\|\psi_{\sigma, t}\|_{W^{2,3}(\pa F+\sigma)}$ is bounded, we infer from \eqref{step61} and standard interpolation estimates that also $\|\psi_{\sigma, t}\|_{C^{1,\beta}(\pa F+\sigma)}$ decays exponentially for $\beta\in (0, \frac13)$.
For all $\beta \in (0,1)$ setting $p = \frac{2}{1 -\beta}$ we have by \eqref{step61} and by \eqref{D(E)}
\beq\label{vdecay}
\begin{split}
\|v_t-v_{F+\sigma}\|_{C^{1,\beta}(\T^3)} &\leq C \|v_t-v_{F+\sigma}\|_{W^{2,p}(\T^3)} \leq C \|u_t-u_{F+\sigma}\|_{L^{p}(\T^3)} \\
&\leq C |E_t \Delta (F+\sigma)|^{\frac{1}{p}}\leq C \|\psi_{\sigma, t}\|_{L^2(\pa F+\sigma)}^{\frac{1}{p}}\\
&\leq C\de_0^{\frac{1}{4p}}\mathrm{e}^{-\frac{c_0}{4p}t}
\end{split}
\eeq
 for all  $\beta\in (0,1)$. Denote the average of $w_t$ on $\pa E_t$ by $\bar w_t$. Since by \eqref{finaldecay} we have that 
$$
\|w_t\big(\cdot + \psi_{\sigma, t}(\cdot)\nu_{F+\sigma}(\cdot)\big)-\bar w_t\|_{H^{\frac12}(\pa F+\sigma)}\leq C
\|w_t-\bar w_t\|_{H^{\frac12}(\pa E_t)} \leq C\|Dw_t\|_{L^2(\T^3)}\leq C\sqrt{\de_0}\mathrm{e}^{-\frac{c_0}2t}\,,
$$
it follows (taking into account also \eqref{vdecay}) that 
\begin{multline}\label{quasiHdecay}
\bigl\|\big[H_t\big(\cdot + \psi_{\sigma, t}(\cdot)\nu_{F+\sigma}(\cdot)\big)-\overline H_t\big]\\
-[H_{\pa F+\sigma}-\overline H_{\pa F+\sigma}]\bigr\|_{H^{\frac12}(\pa F+\sigma)}\to0 \quad\text{exponentially fast,}
\end{multline}
 where $\overline H_t$ and  $\overline H_{\pa F+\sigma}$ stand for the  average of $H_t$ on $\pa E_t$ and of $H_{\pa F+\sigma}$ on $\pa F+\sigma$, respectively.
 Let $d_\sigma$ be the signed distance function from $F+\sigma$ and let $\Psi_t$  denote a diffeomorphism such that $\Psi_t(F+\sigma)=E_t$. Clearly we can find  such a diffeomorphism with the additional property that 
 $\Psi_t(x)=x+\psi_{\sigma, t}(x)\nu_{F+\sigma}(x)$ on $\pa F+\sigma$ and $\|\Psi_t-Id\|_{C^1(\T^3)}\leq C \|\psi_{\sigma, t}\|_{C^{1}(\pa F+\sigma)}$.   Then, denoting  the tangential divergence on $\pa E_t$ by $\Div_{\tau_t}$ and the tangential Jacobian of $\Psi_t$
by $J_\tau\Psi_t$, we have
 \beq\label{quasiquasi}
 \begin{split}
 \biggl|\int_{\pa E_t}H_t\nabla d_\sigma\cdot\nu_t\, d\Ha^2& -\int_{\pa F+\sigma}H_{\pa F+\sigma}\, d\Ha^2\biggr|\vphantom{\Biggl|}\\
 &  =\biggl|\int_{\pa E_t}\Div_{\tau_t}\nabla d_{\sigma}\, d\Ha^2-\int_{\pa F+\sigma}\Div_\tau \nabla d_{\sigma}\, d\Ha^2\biggr|\vphantom{\Biggl|}\\
 &\leq \biggl|\int_{\pa F+\sigma}\bigl(\Div_{\tau_t}\nabla d_{\sigma}\circ\Psi_t J_\tau\Psi_t-\Div_\tau \nabla d_{\sigma}\bigr)\, d\Ha^2\biggr|\vphantom{\Biggl|}\\
 &\leq C \|\psi_{\sigma, t}\|_{C^{1}(\pa F+\sigma)}\,,
 \end{split}
 \eeq
 where the constant $C$ also depends on the $C^2$-bounds on $\pa F$.  Moreover, 
 \beq\label{quasiquasi2}
 \begin{split}
 \biggl|\int_{\pa E_t}(H_t\nabla d_\sigma\cdot\nu_t-H_t)\, d\Ha^2\biggr|& =
  \biggl|\int_{\pa E_t}H_t(\nabla d_\sigma-\nu_t)\cdot\nu_t\, d\Ha^2\biggr|\vphantom{\Biggl|}\\
  &\leq \|H_t\|_{L^1(\pa E_t)}\|\nabla d_\sigma-\nu_t\|_{L^{\infty}(\pa E_t)}\leq C  \|\psi_{\sigma, t}\|_{C^{1}(\pa F+\sigma)}\,,
 \end{split} 
 \eeq
 where we have also used the uniform bounds on $H_t$ established in the previous steps.  Combining \eqref{quasiquasi} and 
 \eqref{quasiquasi2}, we get that  $\overline H_{t}-\overline H_{\pa F+\sigma}$ decays exponentially and in turn, thanks to 
 \eqref{quasiHdecay}
 $$
 \bigl\|H_t\big(\cdot + \psi_{\sigma, t}(\cdot)\nu_{F+\sigma}(\cdot)\big) -H_{\pa F+\sigma}\bigr\|_{H^{\frac12}(\pa F+\sigma)}\to0 \quad\text{exponentially fast.}
 $$
 The conclusion follows arguing  as in  the end of the proof  of Lemma~\ref{w52conv}. 
\end{proof}

Theorem~\ref{main thm 1} can be readily extended to the Neumann case, at least when the stable critical set $F$ is well contained in $\Omega$. Recall in this case the energy \eqref{J} must be replaced with
$$
J_N(E):=P_\Omega(E)+\gamma\int_{\Omega}|\nabla v_E|^2\, dx,
$$
where $P_\Omega(E)$ denotes the perimeter of $E$ inside $\Omega$ and the function $v_E$ is the solution of
$$\begin{cases}
-\Delta v_E=u_E-m\quad \text{ in }\Omega\\
 \displaystyle \int_\Omega v_E\,dx=0\;,\quad \frac{\pa v_E}{\pa\nu}=0\;,\quad \text{on } \pa\Omega \;.\end{cases}$$
 Here $u_E=2\chi_E-1$ and $m=\medintinrigo_\Om u_E\,dx$.
As in \eqref{vEper} we have
$$
v_E(x)=\int_\Omega G(x,y)u_E(y)\, dy\,,$$
where $G$ is the solution of
$$\begin{cases}
-\Delta_y G(x,y)=\delta_x -\frac1{|\Om|}\quad \text{ in }\Omega\\
 \displaystyle \int_\Omega G(x,y)\,dy=0\;,\quad \nabla_y G(x,y)\cdot \nu(y)=0\;,\quad \text{if } y\in\pa\Omega \;.\end{cases}$$

As in the periodic case, we say that a smooth subset $F\subset\!\subset\Om$ is a {\it critical set} for the functional $J_N$ if there exists a constant $\lambda\in\R$ such that
 $$
 H_{\pa F}(x)+4\gamma v_F(x)=\lambda \qquad\text{for all $x\in\pa F$.}
 $$
The quadratic form associated with the second variation $\partial^2J_N(E)$ is also defined as in \eqref{J2}. If $F\subset\!\subset\Om$  is a smooth local minimizer of $J_N$ under volume constraint, then it is also critical and  $\partial^2J_N(E)[\varphi]\geq0$ for all $\varphi\in \widetilde H(\pa F)$.

Note that, unlike in the periodic case, the functional $J_N$ is not translation invariant. Therefore we say that a smooth critical set $F$ is {\it strictly stable} if 
$$
\pa^2J_N(E)[\vphi]>0\qquad\text{for all $\vphi\in \Ht(\pa E)\setminus \{0\}$. }
$$
With these definitions in hand we can state the following counterpart of Theorem~\ref{main thm 1}.
\begin{theorem}\label{mainN}
Let $\Om$ be an open set in $\R^3$ and let $F\subset\!\subset\Om$ be a smooth strictly stable critical set  and  $U$  a regular tubular neighborhood of $\pa F$. Then, for every  $M>0$ and $\alpha\in (0,1)$ there exists $\de_0>0$ with the following property: Let $E_0\in \mathfrak{h}^{2,\alpha}_M(F, U)$ be  such that 
$$
|E_0|=|F|\,, \qquad |E_0\Delta F|\leq \de_0\,, \qquad \text{and}\qquad \int_{\Omega}|Dw_{E_0}|^2\, dx\leq\de_0\,.
$$
 Then,   the unique classical solution $(E_t)_t$ to  the Mullins-Sekerka flow \eqref{MSintro} with initial datum $E_0$ is  defined  for all $t>0$. Moreover,  $E_t\to F$ in 
 $W^{5/2,2}$ exponentially fast as $t\to +\infty$.
\end{theorem}
The proof of this result is similar to the one of Theorem~\ref{main thm 1}. Actually it is simpler since we do not need the argument used in Step 2, where we controlled the translational component of the flow. Note that in the statement of Lemma~\ref{from AFM}  now \eqref{from AFM1}  holds for all $\varphi\in\widetilde H(\pa E)$.  Finally, observe that under the assumptions of Proposition~\ref{prop:nocrit} we may conclude that $F^\prime=F$, i.e., that there are no other critical sets close to $F$.

The assumption that $F$ does not touch the boundary may seem restrictive. However we remark that in two and three dimensions there are examples of strictly stable critical sets  which consist of either a single or multiple almost spherical sets well contained in $\Omega$. The precise conditions on the parameters $m$, $\gamma$ and $|\Omega|$ under which these strictly stable sets exist are given in
\cite{RW1, RW2, RW3}. Other examples of local minimizers well contained in $\Omega$ are given in~\cite{CS}.

\section{Nonlinear stability for the surface diffusion flow}\label{sec:SD}

Throughout the section  we assume $\gamma=0$ in \eqref{J}, so that we will be dealing only with the standard local perimeter. 
We will show how to adapt the strategy devised in the previous one to the case of the surface diffusion equation. 
For the definition of sets of class $h^{2,\alpha}$ we refer to the previous section.

\begin{definition}[Surface diffusion flows]\label{def:SDsol}
 Let $E_0\subset\T^3$ be of class $h^{2,\alpha}$ for some $\alpha\in (0,1)$. We say that  the one-parameter family $(E_t)_{t\in (0, T)}$ is a  {\em classical solution} to the {\em surface diffusion equation}   on the interval $(0, T)$ with initial datum $E_0$ if it is a smooth flow in the sense of  Definition~\ref{def:smoothflow}, $E_t\to E_0$ in $C^{2,\alpha}$ as $t\to 0^+$, and  the following evolution law holds:
\begin{equation}\label{SD}
V_t= \Delta_\tau H_t \quad\text{on } \, \partial E_t\text{ for all }t\in (0, T)\,,
\end{equation}
where, as usual,  $V_t$ stands for the outer normal velocity of the moving boundary $\pa E_t$, $H_t$ stands for $H_{\pa E_t}$ and $\Delta_\tau$ is the Laplace-Beltrami operator on $\pa E_t$.  
\end{definition}
It is well-known that the surface diffusion flow is volume preserving. This can be straightforwardly checked by the following computation:
$$
\frac{d}{dt}|E_t|=\int_{\pa E_t}V_t\, d\Ha^{2}=\int_{\pa E_t}\Delta_\tau H_t\, d\Ha^{2}=0\,. 
$$
 The following local-in-time  existence and uniqueness result has been established in  \cite{EMS}\footnote{In fact \cite{EMS} deals with the evolution in the whole space $\R^N$, but it is clear that the same arguments go through in the periodic case.}. We make use of the notation introduced in the previous section.
 
\begin{theorem}[Local-in-time  existence and uniqueness, \cite{EMS}]\label{th:EMS}
 Let $F_0\subset\T^3$ be a smooth set and   $U$ a regular tubular neighborhood of $\pa F_0$.  Then, for every  $M>0$ and $\alpha\in (0,1)$ there exists $T>0$ with the following property:  For every $E_0\in \mathfrak{h}^{2,\alpha}_M(F_0, U)$
there exists a unique classical solution  to the surface diffusion flow  in $(0, T)$  with initial datum $E_0$.
 \end{theorem}

As before we are interested in the asymptotic stability of strictly stable configurations.
The main result  of the section is the following.
\begin{theorem}[Main result] \label{main thm 2}
Let $F\subset\T^3$ be a strictly stable critical set according to Definition~\ref{def:stability+} and let $U$ be a regular tubular neighborhood of $\pa F$. Then, for every  $M>0$ and $\alpha\in (0,1)$ there exists $\de_0>0$ with the following property: Let $E_0\in \mathfrak{h}^{2,\alpha}_M(F, U)$ be of class $W^{3,2}$  such that 
$$
|E_0|=|F|\,, \qquad |E_0\Delta F|\leq \de_0\,, \qquad \text{and}\qquad \int_{\pa E_0}|D_\tau H_{\pa E_0}|^2\, d\Ha^2\leq\de_0\,.
$$
 Then,   the unique classical solution $(E_t)_t$ to  the surface diffusion flow with initial datum $E_0$ is  defined  for all $t>0$. 
 Moreover,   $E_t\to F+\sigma$ in 
 $W^{3,2}$ as $t\to +\infty$, for some $\sigma\in \R^3$. The convergence is exponentially fast; more precisely, 
 there exist $\eta$, $c_F>0$  such that for all $t>0$, writing 
\[
\pa E_t=\{x+\psi_{\sigma, t}(x)\nu_{F+\sigma}(x):\, x\in \pa F+\sigma \}\,,
\]
we have 
$$
\|\psi_{\sigma, t}\|_{W^{3,2}(\pa F+\sigma)}\leq \eta\mathrm{e^{-c_Ft}}\,.
$$
Both $|\sigma|$ and $\eta$ vanish as $\de_0\to 0^+$.  
\end{theorem}
As before, the proof of the theorem, which is  close in spirit to the proof of Theorem~\ref{main thm 1} is postponed until the end of the section. We first collect some auxiliary results, whose proofs are given in Section~\ref{sec:technicalproofs}.
\begin{lemma}[Energy identities] \label{calculationsbis} Let $(E_t)_{t\in (0, T)}$ be a smooth flow satisfying \eqref{SD}. The  following energy idienties hold:
\begin{equation}
\label{der of Jbis}
\frac{d}{dt} J(E_t) = -  \int_{\pa E_t} |D_\tau H_t|^2\, dx\,,
\end{equation}
and
\begin{equation}
\label{der of DH}
\begin{split}
\frac{d}{dt} \left(\frac{1}{2}\int_{\pa E_t} |D_\tau H_t|^2\, dx \right) = & -\pa^2J(E_t)\left[\Delta_\tau H_t\right]
-\int_{\partial E_t}  B_t \left [D_\tau H_t\right]  \Delta_\tau H_t \, d \Ha^{2}  \\
&+ \frac{1}{2}\int_{\partial E_t} H_t |D_\tau H_t|^2  \Delta_\tau H_t\, d \Ha^{2}\,,
\end{split}
\end{equation}
where $\pa^2J(E_t)$ is the quadratic form defined in \eqref{J2} (with $E_t$ in place of $E$ and with $\gamma=0$) and, as usual, the subscript $t$ stands for  ${E_t}$. Note also that we have used the notation $B_t[\cdot]$ to denote the second fundamental quadratic form on $\pa E_t$, which we recall is defined as $B_t[\tau]:=(D_\tau \nu_t \tau)\cdot\tau$ for all $\tau\in \R^3$.  
\end{lemma}

\begin{lemma}[Interpolation  on boundaries]
\label{interpolation}
Let $F\subset\T^3$ be a smooth set,  $U$   a  regular tubular neighborhood  of $\pa F$, and   $M>0$, $p\in (2,+\infty)$ fixed constants.  
Then, there exists  $C>0$ with the following property:  for every $E
\in \mathfrak{C}^1_M(F, U)$ and   $f \in H^1(\partial E)$  it holds
\[
\|f\|_{L^p(\partial E)} \leq C \left( \|D_\tau f\|_{L^2(\pa E)}^{\theta} \|f\|_{L^2(\pa E)}^{1-\theta} + \|f\|_{L^2(\pa E)} \right)\,,
\]
with  $\theta: = 1- \frac{2}{p}$.  Moreover, the following Poincar\'e inequality holds
\[
\|f-\bar f\|_{L^p(\pa E)} \leq C \|D_\tau f\|_{L^2(\partial E)}, 
\]
where $\bar f$ denotes the piecewise constant function defined as $\medintinrigo_{\Gamma^i} f\, d\Ha^2$ on each connected component  $\Gamma^i$ of $\partial E$.
\end{lemma}
The proof of the above lemma can be found in \cite[Theorem 3.70]{aubin}.

For the next lemma we introduce the following notation: for every sufficiently regular $f$ defined on $\pa E$  we set
\begin{equation}\label{kompelot derivaatat}
\delta_i f := D_\tau f \cdot e_i \qquad \text{and} \qquad  \qquad D_\tau^2 f := (\delta_i \delta_j f  )_{i,j},
\end{equation}
where $e_i$ is the $i$-th element of the canonical basis  of $\R^3$.  
\begin{lemma}[$H^2$-estimates on boundaries]
\label{laplacian}
Let $F$, $U$, and $M$ be as  in Lemma ~\ref{interpolation}. Then there exists a constant $C>0$ such that if $E\in \mathfrak{C}^1_M(F, U)$ and   $f\in H^1(\pa E)$, with 
$\Delta_\tau f\in L^2(\pa E)$, then $f\in H^2(\pa E)$ and  
\[
\|D_\tau^2 f\|_{L^2(\partial E)}\leq C  \|\Delta_\tau f\|_{L^2(\pa E)}(1+  \|H_{\pa E}\|_{L^4(\pa E)}^2).
\]
\end{lemma}
The following lemma provides the crucial ``geometric interpolation'' that will be needed in the proof of the main theorem. 
\begin{lemma}[Geometric interpolation]
\label{nasty}
Let $F$, $U$, and $M$ be as  in Lemma~\ref{interpolation}. There exists a constant $C>0$ such that if $E\in \mathfrak{C}^1_M(F, U)$
the following estimates holds:
\begin{multline*}
\int_{\partial E} |B_{\pa E}| |D_\tau H_{\pa E}|^2 |\Delta_\tau H_{\pa E}| \, d \Ha^{2} \\
 \leq C \| D_\tau (\Delta_\tau H_{\pa E})\|_{L^2(\pa E)}^2 \,  \|D_\tau H_{\pa E}\|_{L^2(\pa E)}\, \left(1+ \|H_{\pa E}\|_{L^6(\pa E)}^3 \right)\,.
\end{multline*}
\end{lemma}

The next lemma highlights an interesting property of the mean curvature. Note that since $\partial E$ can be disconnected (as in the case of lamellae) one can not expect 
Poincar\'e inequality to hold on  $\partial E$. However, if $E$ is sufficiently close to a stable critical set then the Poincar\'e inequality holds for $H_{\pa E}$. 

\begin{lemma}[Geometric Poincar\'e Inequality]\label{lm:geopoinc}
Fix $p>2$,  let $F\subset\T^3$ be a strictly stable critical set according to 
Definition~\ref{def:stability+} and let $\de_1$ be the constant provided by Lemma~\ref{from AFM}, with   $\e=1$ (and $\gamma=0$). Then, there exists $C>0$ such that 
\beq\label{geopoinc}
\int_{\pa E}|H_{\pa E}-\overline H_{\pa E}|^2\, d\Ha^2\leq C\int_{\pa E}|D_\tau H_{\pa E}|^2\, d\Ha^2\,,
\eeq
provided that 
$$
\pa E=\{x+\psi(x) \nu_F(x):\, x\in \pa F \text{ for some smooth $\psi$ with $\|\psi\|_{W^{2,p}(\pa F)}\leq \de_1$}\}.  
$$
Here $\overline H_{\pa E}$ stands for the average $\medintinrigo_{\pa E}H_{\pa E}\, d\Ha^2$.
\end{lemma}

Finally, we have:

\begin{lemma}[Compactness of sets]\label{w32conv}
Let $F$, $U$, and $M$ be as  in Lemma ~\ref{interpolation}. Let $\{E_n\}_n\subset \mathfrak{C}^1_M(F, U) $ be a sequence of sets  such that 
$$
\sup_n\int_{\pa E_n}|D_\tau H_{\pa E_n}|^2\, dx<+\infty\,.
$$
Then there exists $F'\in \mathfrak{C}^1_M(F, U)$ of class  $W^{3,2}$ such that, up to a  (non relabeled) subsequence, $E_n\to F'$ in $W^{2,p}$ for all $p\in [1,+\infty)$. Moreover, if
\eqref{geopoinc} holds for every set $E_n$ (with $C$ independent of $n$) and 
$$
\int_{\pa E_n}|D_\tau H_{\pa E_n}|^2\, dx\to 0\,,
$$
then $F'$ is critical in the sense of Definition~\ref{def:criticality} and the convergence holds in $W^{3,2}$.
\end{lemma}
The proof of this lemma is similar to the proof of Lemma~\ref{w52conv} given in Subsection~\ref{subsec:SD} and thus we omit it. 

\begin{proof}[\textbf{Proof of Theorem~\ref{main thm 2}.}] The proof of the theorem is very close in spirit to the proof of 
Theorem~\ref{main thm 1}. 
In the following,  $C$ will denote a constant depending only on the $C^{1}$-bounds on the boundary of the set. The value of $C$  may change from line to line.  
  For every  $\e_0>0$ sufficiently small, there exists $\de_0\in (0,1)$ so small that for any set $E\in \mathfrak{C}^{1}_M(F, U)$ the following implications hold true:
 \beq\label{de01bis}
 E\in \mathfrak{h}^{2,\alpha}_M(F, U)\text{ and } D(E)\leq \de_0\Longrightarrow \|\psi_E\|_{C^1(\pa F)}\leq \frac{\e_0}2\,,
 \eeq  
where $D(E)$ is defined in \eqref{D(E)0}, and 
 \beq\label{de02bis}
  \|\psi_E\|_{C^1(\pa F)}\leq \e_0\text{ and } \int_{\pa E} |D_\tau H_{\pa E}|^2\, d\Ha^2 \leq 1\Longrightarrow \|\psi_E\|_{W^{2,6}(\pa F)}\leq \omega(\e_0)\leq 1\,,
 \eeq  
 where $\omega$ is a positive non-decreasing function such that $\omega(\e_0)\to 0$ as $\e_0\to 0^+$. Note that the last implication is true thanks to Lemma~\ref{w32conv}. 
 
 Note also that by Lemma~\ref{lm:geopoinc}, there exists $C>0$ such that if  $\e_0$ is small enough, then
 \begin{multline}\label{de02bisbis}
 \|\psi_E\|_{W^{2,6}(\pa F)}\leq \omega(\e_0)
 \Longrightarrow 
 \int_{\pa E}|H_{\pa E}-\overline H_{\pa E}|^2\, d\Ha^2\leq C\int_{\pa E}|D_\tau H_{\pa E}|^2\, d\Ha^2\,,
 \end{multline}
where $\overline H_{\pa E}$ is the average of $H_{\pa E}$ over $\pa E$.
 Fix $\e_0$, $\de_0\in (0,1)$ satisfying \eqref{de01bis}, \eqref{de02bis} and \eqref{de02bisbis}, and choose an initial set $E_0\in \mathfrak{h}^{2,\alpha}_M(F, U)$ such that
 \beq\label{initialbis}
 D(E_0)\leq \delta_0\qquad\text{and}\qquad \int_{\pa E_0} |D_\tau H_{\pa E_0}|^2\, d\Ha^2 \leq \delta_0\,.
 \eeq
 Let $(E_t)_{t\in (0, T(E_0))}$ be the unique classical solution to the surface diffusion flow provided 
 by Theorem~\ref{th:EMS}, with $T(E_0)$ denoting the maximal time of existence. By the same theorem, there exists $T_0>0$ such that  \eqref{dalbasso} holds.
 We now split the rest of the proof into several steps as in the proof of Theorem~\ref{main thm 1}. 
 
 \vspace{5pt}
\noindent {\bf Step 1.}{\it (Stopping-time)} Let $\bar t\leq T(E_0)$ be the maximal time such that 
\beq\label{Tprimobis}
\|\psi_t\|_{C^1(\pa F)}<\e_0\quad\text{and}\quad \int_{\pa E_t} |D_\tau H_{t}|^2\, d\Ha^2 < 2\delta_0.
\quad\text{for all $t\in (0, \bar t)$,}
\eeq
 As before,  we claim that  by taking $\eps_0$ and $\de_0$ smaller if needed, we have $\bar t=T(E_0)$.
 
\vspace{5pt}

\vspace{5pt}
\noindent {\bf Step 2.}{\it (Estimate of the  translational component of the flow)} We claim that there exists $\e>0$ such that
\begin{equation} \label{not a translationbis}
\min_{\eta\in \Pi_F}\big\|\Delta_\tau H_t-\eta\cdot\nu_t\big\|_{L^2(\pa E_t)}\geq \e\|\Delta_\tau H_t\|_{L^2(\pa E_t)}\qquad\text{for all }t\in (0, \bar t)\,,
\end{equation}
where $\Pi_F$ is defined in \eqref{projektio}. To this aim, let $\eta_t\in \Pi_F$ be such that 
\begin{equation} \label{not a translation 2bis}
\Delta_\tau H_t  = \eta_t\cdot \nu_t + g,
\end{equation} 
where $g$ is orthogonal to the subspace of $L^2(\partial E_t)$ spanned by $\tilde e_i\cdot \nu_{t}$ with $i \in I_F$ (see \eqref{indeksit}). As in Step 2 of the proof of Theorem~\ref{main thm 1} we will show that if $\e$ is small enough, then assuming $\|g\|_{L^2(\pa E_t)} < \eps \| \Delta_\tau H_t  \|_{L^2(\pa E_t)}$ leads to a contradiction. Recall that $\Delta_\tau H_t$ has zero average. Therefore, 
setting $\overline H_t:=\medintinrigo_{\pa E_t}H_t\, d\Ha^2$,  and recalling also \eqref{de02bis} and  \eqref{de02bisbis}, we get 
\beq\label{sfiguz}
\begin{split}
\|H_t -\overline H_t \|_{L^2(\pa E_t)}^2 & \leq C \int_{\pa E_t}  |D_\tau H_t|^2\, d\Ha^2 \\
&= -C\int_{\pa E_t} \Delta_\tau H_t  H_t\, d\Ha^2= -C\int_{\pa E_t}  \Delta_\tau H_t  (H_t-\overline H_t )\, d\Ha^2\\
&\leq C \|H_t -\overline H_t \|_{L^2(\pa E_t)}\| \Delta_\tau H_t\|_{L^2(\pa E_t)}\,.    
\end{split}
\eeq
Recall now that $\int_{\pa E_t}H_t\nu_t\, d\Ha^2=\int_{\pa E_t}\nu_t\, d\Ha^2=0$.  Thus,  multiplying \eqref{not a translation 2bis} by $H_t-\overline H_t$, integrating over $\partial E_t$, and using \eqref{sfiguz},  we get
\[
\begin{split}
\biggl| \int_{\partial E_t}  (H_t -\overline H_t)\Delta_\tau H_t   \, d \Ha^{2}  \biggl|&= 
\biggl| \int_{\partial E_t}  (H_t -\overline H_t)g  \, d \Ha^{2}  \biggl|\\
&< \eps \| H_t -\overline H_t \|_{L^2(\partial E_t)} \| \Delta_\tau H_t  \|_{L^2(\pa E_t)}\\
&\leq C\e \| \Delta_\tau H_t  \|_{L^2(\pa E_t)}^2\,.
\end{split}
\]
Arguing as in Step 2 of the proof of Theorem~\ref{main thm 1} we have that, if $\eps_0$ is small enough there exists a constant $C$ such that 
 $|\eta_t|\leq C\| \Delta_\tau H_t  \|_{L^2(\pa E_t)}$. Hence
\[
\begin{split}
\|\eta_t\cdot\nu_t\|^2_{L^2(\pa E_t)}&=\int_{\pa E_t}\Delta_\tau H_t (\eta_t\cdot\nu_t)\, d\Ha^2= 
- \int_{\partial E_t}   D_\tau H_t \cdot    D_\tau (\eta_t\cdot\nu_t)\, d \Ha^{2}\\
&\leq|\eta_t|\|D_\tau\nu_t\|_{L^2(\pa E_t)} \| D_\tau H_t  \|_{L^2(\pa E_t)}\\
&\leq C \|D_\tau\nu_t\|_{L^2(\pa E_t)} \| \Delta_\tau H_t  \|_{L^2(\pa E_t)} \left(- \int_{\partial E_t} (H_t -\overline H_t)\Delta_\tau H_t \, d \Ha^{2} \right)^{1/2}\\
&\leq C \|D_\tau\nu_t\|_{L^2(\pa E_t)}  \eps^{1/2}\| \Delta_\tau H_t  \|_{L^2(\pa E_t)}^2\leq C  \eps^{1/2}\| \Delta_\tau H_t  \|_{L^2(\pa E_t)}^2\,,
\end{split}
\]
where in the last inequality the constant $C$ depends also on the curvature bounds provided by \eqref{de02bis}.
If $\e$ is chosen so small that $C\e^{\frac12}+\e^2 < 1$ in the last inequality, then we reach a contradiction to 
\eqref{not a translation 2bis} and the fact that $\|g\|_{L^2(\pa E_t)}<\e \| \Delta_\tau H_t  \|_{L^2(\pa E_t)}$. 

As in Step 2 of the proof of  Theorem~\ref{main thm 1}, by taking $\e_0$ (and $\de_0$) smaller if needed, we may ensure that \eqref{de05} holds,
 with $\omega$ the modulus of continuity introduced in \eqref{de02bis} and $\de_1$ satisfying \eqref{de03} and \eqref{de04}, with $W^{2,3}(\pa F)$ replaced by $W^{2,6}(\pa F)$.
 
 \vspace{5pt}
\noindent {\bf Step 3.}{\it (The stopping time $\bar t$ equals the maximal time $T(E_0)$)} Here we  assume by contradiction that $\bar t< T(E_0)$ and thus  
$$
\|\psi_{\bar t}\|_{C^1(\pa F)}=\e_0\quad\text{or}\quad \int_{\pa E_{\bar t}} |D_\tau H_{\bar t}|^2\, d\Ha^2 =2\delta_0\,.
$$
We further split into two sub-steps, according to the two alternatives above.

\vspace{5pt}
{\it Step 3-(a).}  Assume that 
\beq\label{step3abis}
\int_{\pa E_{\bar t}} |D_\tau H_{\bar t}|^2\, d\Ha^2 =2\delta_0\,.
\eeq
Recall that \eqref{not a translationbis} holds. Thus, by \eqref{de02bis}, \eqref{Tprimobis},  \eqref{de03} (with $W^{2,3}(\pa F)$ replaced by $W^{2,6}(\pa F)$), and \eqref{de05} we have
$$
\partial^2 J(E_t)\left[ \Delta_\tau H_t\right]\geq \sigma_\e\| \Delta_\tau H_t \|_{H^1(\pa E)}^2 \text{ for all $t\in (0, \bar t)$.}
$$
Note also that \eqref{sfiguz}, together with the Poincar\'e inequality \eqref{geopoinc}, yields
\beq\label{sfiguz2}
\| D_\tau H_t\|_{L^2(\pa E_t)}\leq C\| \Delta_\tau H_t\|_{L^2(\pa E_t)}\,.
\eeq 
Now, we may use Lemma~\ref{calculationsbis} to estimate 
\[
\begin{split}
\frac{d}{dt} \left(\frac{1}{2} \int_{\partial E_t} |D_\tau H_t|^2\, d \Ha^{2} \right) &\stackrel{\hphantom{Lemma~\ref{nasty}}}{\leq}  - \sigma_{\eps} \|\Delta_\tau H_t\|^2_{H^1(\pa E_t)}  + 2\int_{\partial E_t} |B_t| | D_\tau H_t |^2 |\Delta_\tau H_t| \, d \Ha^{2} \\
&\stackrel{Lemma~\ref{nasty}}{\leq} - \sigma_{\eps} \|\Delta_\tau H_t\|^2_{H^1(\pa E_t)} \\
&\qquad  + C \| D_\tau (\Delta_\tau H_t)\|_{L^2(\pa E_t)}^2 \|D_\tau H_t\|_{L^2(\pa E_t)} \left(1+ \|H_t\|_{L^6(\pa E_t)}^3\right) \\
&\stackrel{\hphantom{Le}\eqref{Tprimobis}\quad}{\leq}  - \sigma_{\eps} \|\Delta_\tau H_t\|^2_{H^1(\pa E_t)} \\
&\qquad  + C \sqrt{\de_0}\| D_\tau (\Delta_\tau H_t)\|_{L^2(\pa E_t)}^2  \left(1+ \|H_t\|_{L^6(\pa E_t)}^3\right) \\
&\stackrel{\hphantom{Le}\eqref{de02bis}\quad}{\leq}  - \sigma_{\eps} \|\Delta_\tau H_t\|^2_{H^1(\pa E_t)}   + C \sqrt{\de_0}\| D_\tau (\Delta_\tau H_t)\|_{L^2(\pa E_t)}^2
\end{split}
\]
for every $t \leq \bar t$. Thus, if we choose $\de_0$ small enough we have
$$
\frac{d}{dt} \left(\frac{1}{2} \int_{\partial E_t} |D_\tau H_t|^2\, d \Ha^{2} \right)\leq    - \frac{\sigma_{\eps}}2 \|\Delta_\tau H_t\|^2_{H^1(\pa E_t)}\leq -c_0 \|D_\tau H_t\|^2_{L^2(\pa E_t)}\,,
$$
where the last inequality follows from \eqref{sfiguz2}.

  Integrating the differential inequality  and recalling \eqref{initialbis}, we obtain
\beq\label{expfinalmentebis}
\int_{\partial E_t} |D_\tau H_t|^2\, d \Ha^{2} \leq \mathrm{e}^{-c_0 t}\int_{\partial E_0} |D_\tau H_{E_0}|^2\, d \Ha^{2} \leq \delta_0  \mathrm{e}^{-c_0 t}
\eeq
which gives a contradiction to \eqref{step3abis} for $t=\bar t$.

\vspace{5pt}
{\it Step 3-(b).}  Assume now that 
\beq\label{step3bbis}
\|\psi_{\bar t}\|_{C^1(\pa F)}=\e_0\,.
\eeq
Then, arguing as in Step 3-(b) of the proof of Theorem~\ref{main thm 1}, we can compute
\[
\begin{split}
\frac{d}{dt}D(E_t)&= \int_{E_t}\Div(d_F X_t)\, dx= \int_{\pa E_t}d_F\, \Delta_\tau H_t\, d\Ha^2\\
&=- \int_{\pa E_t}D_\tau d_F\cdot D_\tau H_t\, d\Ha^2\leq C \|D_\tau H_t\|_{L^2(\pa E_t)}\leq C\sqrt{\de_0} \mathrm{e}^{-\frac{c_0}2 t}\,,
\end{split}
\]
where the last inequality clearly follows from \eqref{expfinalmentebis}. We may now argue exactly as  in the end of Step 3-(b) of the proof of Theorem~\ref{main thm 1} and reach  a contradiction to  \eqref{step3bbis} if $\de_0$ is small enough.

Thus $\bar t=T(E_0)$, and as a byproduct of   \eqref{expfinalmentebis} and of Step 3-(b)  we also have
\begin{multline}\label{finaldecaybis}
\|\psi_t\|_{C^1(\pa F)}<\e_0\\
\text{ and }\int_{\partial E_t} |D_\tau H_t|^2\, d \Ha^{2} \leq \mathrm{e}^{-c_0 t}\int_{\partial E_0} |D_\tau H_{E_0}|^2\, d \Ha^{2} \quad\text{ for all $t\in (0, T(E_0))$.}
\end{multline}

\vspace{5pt}
\noindent {\bf Step 4.}{\it (Global-in-time existence)} Here we assume by contradiction $T(E_0)<+\infty$. Then, we may argue exactly as in Step 4 of the proof of Theorem~\ref{main thm 1} to find  $\hat t\in \left(T(E_0)\!-\!\frac{T_0}2,T(E_0)\!-\!\frac{T_0}4 \right)$ such that
 $\|\Delta_\tau H_{\hat t} \|_{H^1(\pa E_{\hat t})}^2\leq \frac{8\de_0}{T_0\sigma_\e}$. Thus, also by Lemma~\ref{laplacian}
 $$
 \|D^2_\tau H_{\hat t}\|_{L^2(\pa E_{\hat t})}^2\leq C\|\Delta_\tau H_{\hat t} \|^2_{L^2(\pa E_{\hat t})}\left(1+\|H_{\hat t}\|^4_{L^4(\pa E_{\hat t})}\right)\leq C\de_0\,,
 $$
 where in the last inequality we also used the curvature bounds provided by \eqref{de02bis}. In turn, for $p$ large enough
 $$
  [H_{\hat t}]^2_{C^{0,\alpha}(\pa E_{\hat t})}\leq C\|D_\tau H_{\hat t}\|^2_{L^p(\pa E_{\hat t})}\leq C \|D_\tau H_{\hat t}\|^2_{H^1(\pa E_{\hat t})}\leq C\de_0\,,
 $$
 where in the last equality we used also \eqref{finaldecaybis}.
 
  Thus, if we choose $\de_0$ sufficiently small, then 
 $E_{\hat t}\in \mathfrak{h}^{2,\alpha}_M(F, U)$ and, by \eqref{dalbasso} the   time span   of existence of the classical solution starting from  $E_{\hat t}$ is at least $T_0$.  This implies that  $(E_t)_t$ can be continued beyond $T(E_0)$, leading to  a contradiction. 
 
We can now proceed exactly as in Steps 5 and 6 of the proof of Theorem~\ref{main thm 1}, using Lemma~\ref{w32conv} instead of Lemma~\ref{w52conv}, to get the desired conclusion. We leave the details to the reader. 
\end{proof}

\section{Proofs of technical lemmas}\label{sec:technicalproofs}
In this final section  we  collect the proofs of the several technical lemmas stated in the previous sections. 
\subsection{The modified Mullins-Sekerka flow: proof of technical lemmas}\label{subsec:MSnl}
\begin{proof}[Proof of Lemma~\ref{from AFM}]
\noindent{\bf Step 1.} First we claim that the strict stability of $F$ (Definition~\ref{def:stability+}) implies
\begin{equation} \label{uusi stability}
\partial^2J(F)[\vphi]>0 \qquad \text{for all }\, \vphi \in \widetilde H(\pa F) \setminus T(\pa F).
\end{equation}
To this aim we observe that from \eqref{vEper} we get
\[
Dv_F(x) = 2 \int_F D_x G_{\T^3}(x,y) \, dy = - 2\int_F D_y G_{\T^3}(x,y) \, dy = -2 \int_{\partial F}  G_{\T^3}(x,y) \nu(y) \, d \Ha^2(y). 
\]
Setting $\nu_i = e_i \cdot \nu_F$ we have by \cite[Lemma 10.7]{Giu}
\[
-\Delta_\tau \nu_i - |B_{\partial F}|^2\nu_i = -\delta_i H_{\partial F}
\]
where $\delta_i $ is defined as in \eqref{kompelot derivaatat}. Since  $F$ is critical it satisfies $H_{\partial F} + 4 \gamma v_F= const.$ and by the above identities, we have
\[
-\Delta_\tau \nu_i - |B_{\partial F}|^2\nu_i = -4\gamma \partial_\nu v_F \nu_i -8\gamma \int_{\partial F}  G_{\T^3}(x,y) \nu_i(y) \, d \Ha^2(y).
\]
This can be written as $L(\nu_i)= 0$,   where $L: H^1(\partial F) \to H^{-1}(\partial F)$ is self-adjoint, linear operator defined as
\[
L(\vphi) := -\Delta_\tau \vphi - |B_{\partial F}|^2\vphi  + 4\gamma \partial_\nu v_F \vphi + 8\gamma \int_{\partial F}  G_{\T^3}(x,y) \vphi(y) \, d \Ha^2(y).
\]
Let now $\vphi \in \widetilde H(\pa F) \setminus T(\pa F)$. We may write $\vphi = \psi + \eta \cdot  \nu_F$ for some $\eta \in \R^3$, where $\psi \in T^\perp(\partial F)\setminus\{0\}$.
Since $L$ is self-adjoint, we then conclude 
\[
\begin{split}
&\partial^2J(F)[\vphi] =\langle L(\vphi), \vphi \rangle_{H^{-1}\times H^1}\\
&= \langle L(\psi), \psi \rangle_{H^{-1}\times H^1}+ 2 \langle L(\eta \cdot  \nu_F), \psi\rangle_{H^{-1}\times H^1} +\langle L(\eta \cdot  \nu_F), \eta \cdot  \nu_F   \rangle_{H^{-1}\times H^1} 
= \partial^2J(F)[\psi] >0, 
\end{split}
\]
where the last inequality follows from the strict stability assumption on $F$.

Having proved \eqref{uusi stability} we show next that for every  $\e\in (0, 1]$ it holds
\begin{multline}\label{c0}
m_\e:=\inf\Bigl\{\pa^2J(F)[\vphi]:\, \vphi\in \Ht(\pa F)\,, \|\vphi\|_{H^1(\pa F)}=1\\
\text{ and }\min_{\eta\in \Pi_F}\|\vphi-\eta\cdot\nu_F\|_{L^2(\pa F)}\geq \e\|\vphi\|_{L^2(\pa F)}\Bigr\}>0\,.
\end{multline}
Indeed, let $\vphi_h$ be a minimizing sequence for the infimum in \eqref{c0} and assume that $\vphi_h\wto \vphi_0\in \Ht(\pa F)$ weakly in $H^1(\pa F)$. Let us first assume that $\vphi_0\neq 0$. 
Since 
$$
\min_{\eta\in  \Pi_F}\|\vphi_0-\eta\cdot\nu_F\|_{L^2(\pa F)}\geq \e\|\vphi_0\|_{L^2(\pa F)},
$$
we conclude  $\vphi_0\in \Ht(\pa F)\setminus T(\pa F)$.
Thus,
$$
m_\e=\lim_h \pa^2 J(F)[\vphi_h]\geq \pa^2 J(F)[\vphi_0]>0\,,
$$
where the last inequality follows from \eqref{uusi stability}. If $\vphi_0= 0$, then
$$
m_\e=\lim_h\pa^2 J(F)[\vphi_h]=\lim_h\int_{\pa F}|D_\tau \vphi_h|^2 \, d\Ha^2=1\,.
$$ 
 
  \noindent {\bf Step 2.}   In order to conclude the proof of the lemma it is enough to show the existence of $\delta>0$ such that
 if $\pa E=\{x+\psi(x) \nu_F
(x):\, x\in \pa F\}$ with  $\|\psi\|_{W^{2,p}(\pa F)}\leq\delta$, then
\begin{multline}\label{c2unif}
\inf\Bigl\{\pa^2J(E)[\vphi]:\, \vphi\in \Ht(\pa E)\,, \|\vphi\|_{H^1(\pa E)}=1\\
\text{ and }\min_{\eta\in  \Pi_F}\|\vphi-\eta\cdot\nu_E\|_{L^2(\pa E)}\geq \e\|\vphi\|_{L^2(\pa E)}\Bigr\}\geq\sigma_\e:=\frac{1}{2}\min \{ m_{\e/2},1\}\,,
\end{multline}
where $m_{\e/2}$ is defined in \eqref{c0}, with $\e/2$ in place of $\e$. 
Assume by contradiction that there exist a sequence $E_h$, with $\pa E_h=\{x+\psi_h(x) \nu_F
(x):\, x\in \pa F\}$ and $\|\psi_h\|_{W^{2,p}(\pa F)}\to 0$, and a sequence  $\vphi_h\in \Ht(\pa E_h)$, with $\|\vphi_h\|_{H^1(\pa E_h)}=1$ and $\min_{\eta\in  \R^3}\|\vphi_h-\eta\cdot\nu_{E_h}\|_{L^2(\pa E_h)}\geq \e\|\vphi_h\|_{L^2(\pa E_h)}$, such that 
\beq\label{liminf}
\pa^2J(E_h)[\vphi_h]<\sigma_\e\,.
\eeq
Assume first that $\lim_{h}\|\vphi_h\|_{L^2(\pa E_h)}=0$ and observe that by  Sobolev embedding  $\|\vphi_h\|_{L^q(\pa E_h)}\to0$ for every $q>1$. Thus, since $\psi_h$ are uniformly bounded in $W^{2,p}$ for $p>2$ 
we obtain
\[
\lim_h\pa^2 J(E_h)[\vphi_h] = 1,
\]
which is a contradiction to \eqref{liminf}.

Thus we may assume that 
\beq\label{lim positive}
\lim_{h}\|\vphi_h\|_{L^2(\pa E_h)} >0. 
\eeq
The idea now is to read  $\vphi_h$ as a function on $\pa F$. For $x\in \pa F$ set
$$
\tilde\vphi_h(x):=\vphi_h\bigl(x+\psi_h(x) \nu_F(x)\bigr)- \medint_{\pa F}\vphi_h(y+\psi_h(y) \nu_F(y))\, d\Ha^{2}(y)\,.
$$
 As $\psi_h\to 0$ in $W^{2,p}(\pa F)$, we have in particular that 
 \beq\label{sothat}
 \tilde\vphi_h\in \Ht(\pa F)\,, \qquad \|\tilde\vphi_h\|_{H^1(\pa F)}\to 1\,, \qquad\text{and}
 \qquad \frac{\|\tilde \vphi_h\|_{L^2(\pa F)}}{\| \vphi_h\|_{L^2(\pa E_h)}}\to 1\,.
 \eeq
Note also that $\nu_{E_h}(\cdot +\psi_h(\cdot)\nu_F(\cdot))\to \nu_F$ in $W^{1,p}(\pa F)$ and thus in $C^{0,\alpha}(\pa F)$ for a suitable $\alpha\in (0,1)$ depending on $p$. Using also this, and taking into account the third limit in \eqref{sothat} and \eqref{lim positive}, one can easily show that 
$$
\liminf_h\frac{\min_{\eta\in  \Pi_F}\|\tilde \vphi_h-\eta\cdot\nu_{F}\|_{L^2(\pa F)}}{\|\tilde \vphi_h\|_{L^2(\pa F)}}\geq 
\liminf_h\frac{\min_{\eta\in  \Pi_F}\|\vphi_h-\eta\cdot\nu_{E_h}\|_{L^2(\pa E_h)}}{\|\vphi_h\|_{L^2(\pa E_h)}}\geq \e\,.
$$
Thus, for $h$ large enough we have 
$$
\|\tilde\vphi_h\|_{H^1(\pa F)}\geq \frac34\qquad\text{and}\qquad \min_{\eta\in  \Pi_F}\|\tilde \vphi_h-\eta\cdot\nu_{F}\|_{L^2(\pa F)}\geq \frac{\e}2\|\tilde \vphi_h\|_{L^2(\pa F)}\,.
$$
In turn, by Step 1 we infer
\beq\label{bystep1}
\pa^2J(F)[\tilde \vphi_h]\geq \frac{9}{16}m_{\e/2}\,.
\eeq
Moreover, the $W^{2,p}$ convergence of $E_h$ to $F$ and  standard elliptic estimates for the problem \eqref{eqvE} imply
\beq\label{proj2}
B_{\pa E_h}\bigl(\cdot +\psi_h(\cdot) \nu_F(\cdot)\bigr)\to B_{\pa F}\ \text{in  $L^p(\pa F)$},\qquad v_{E_h}\to v_F\ \text{in $C^{1,\beta}(\T^3)$ for all $\beta<1$.}
\eeq
We now check that
\begin{multline}\label{proj3}
\int_{\partial E_h}\!\int_{\partial E_h}\!\!\!G_{\T^3}(x,y)\vphi_h(x)\vphi_h(y)\, d\Ha^{2}(x) d\Ha^{2}(y)\\ -
\int_{\partial F}\!\int_{\partial F}\!\!G_{\T^3}(x,y)\tilde\vphi_h(x)\tilde\vphi_h(y)\, d\Ha^{2}(x) d\Ha^{2}(y)\to 0
\end{multline}
as $h\to\infty$. Indeed, thanks to Remark~\ref{rm:potential} this is equivalent to  
\beq\label{differenze}
\int_\Om\bigl(|D z_{h}|^2-|D \tilde z_{h}|^2\bigr)\, dz\to 0\,,
\eeq
where 
$$
-\Delta z_h=\mu_h:=\vphi_h\Ha^2\rtangle\pa E_h\,,\qquad -\Delta \tilde z_h=\tilde\mu_h:=\tilde\vphi_h\Ha^2\rtangle\pa F\,,
$$
under periodicity condition.  In turn, \eqref{differenze}   is clearly implied by 
$$
\mu_h-\tilde\mu_h\to 0 \quad\text{in $H^{-1}(\T^3)$,}
$$
which can be easily checked (see \cite[Proof of Theorem~3.9]{AFM} for the details). 

Finally, we observe that since $p>2$, the Sobolev Embedding theorem and the $W^{2,p}$-convergence of $\pa E_h$ to $\pa F$ imply
\beq\label{proj1000}
 \int_{\pa E_h}|B_{\pa E_h}|^2\vphi_h^2\, d\Ha^2-\int_{\pa F}|B_{\pa F}|^2\tilde \vphi_h^2\, d\Ha^2\to 0\,.
 \eeq
Combining  \eqref{proj2}, \eqref{proj3}, and \eqref{proj1000} we  conclude that all terms 
 of $\pa^2J(E_h)[\vphi_h]$ are asympotically close to the corresponding terms of $\pa^2J(E)[\tilde\vphi_h]$ and thus
$$
\pa^2J(E_h)[\vphi_h]-\pa^2J(F)[\tilde\vphi_h]\to 0\,.
$$
Recalling \eqref{liminf},  we have  a contradiction to \eqref{bystep1}.  This establishes \eqref{c2unif} and concludes the proof of the lemma.
\end{proof}

\begin{proof}[Proof of Lemma~\ref{calculations}]
In the following  $\Psi$ and $\Psi_t$ are as in Definition~\ref{def:smoothflow} and the subscript $t$ stands for the subscript $E_t$. 
We denote by $X_t$ the associated velocity field, that is, $X_t:=\dot\Psi_t\circ\Psi_t^{-1}$.
In particular, by \eqref{MSnl} we have that 
\beq\label{xtnu}
X_t\cdot \nu_t=[\pa_{\nu_t}w_t]\qquad\text{ on $\pa E_t$.}
\eeq
Fix $t\in (0,T)$,  set $\Phi_s:=\Psi_{t+s}\circ \Psi_t^{-1}$, and note that  $(\Phi)_{s\in (-t, T-t)}$ is an admissible  one-parameter family  of diffeomorphisms according to Definition~\ref{def:admissibleX}.  Then we may apply Theorem~\ref{th:12var} to get
\begin{align*}
\frac{d}{dt}J(E_t)&\stackrel{\phantom{\eqref{WE}}}{=}\frac{d}{ds}J(\Phi_s(E_t))_{\bigl|_{s=0}}=\int_{\pa E_t}(H_{t}+4\gamma v_t)X_t\cdot \nu_t\, d\Ha^{2} \\
&\stackrel{\eqref{WE}}{=} \int_{\pa E_t}w_t\, X_t\cdot \nu_t\, d\Ha^{2}
\stackrel{\eqref{xtnu}}{=} \int_{\pa E_t}w_t [\pa_{\nu_t}w_t]\, d\Ha^{2} \\
&\stackrel{\phantom{\eqref{WE}}}{=}-\int_{\T^3}|Dw_t|^2\, dx\,,
\end{align*}
where the last equality follows from integration by parts and the fact that $w_t$ is harmonic in $\T^3\setminus \pa E_t$. This establishes \eqref{der of J}.
In order to get \eqref{der of dw}, we need to introduce some auxiliary functions: For each $t\in (0, T)$, we let $d_{t}$ denote the signed distance function from $E_t$, which, we recall, is smooth in a suitable tubular neighborhood of $\pa E_t$.  We then set $\nu_t:=Dd_t$, $H_t:=\Delta d_t=\Div\nu_t$, and $B_t:=D^2d_t= D\nu_t$. Note that    $\nu_t$, $H_t$, and $B_t$ represent smooth extensions of  the outer unit normal field,  the mean curvature  and the second fundamental form, respectively,  to a    neighborhood of $\pa E_t$.
 We start by recalling the following identity (see \cite[Lemma 3.8]{CMM}):
\beq\label{dnuacca}
\partial_{\nu_t} H_t = DH_t\cdot \nu_t = - |B_t|^2 \qquad\text{on $\pa E_t$} 
\eeq
and
\beq\label{nupunto}
\dot{\nu}_t:=\frac{\partial}{\partial s}\nu_{t+s}\Bigl|_{s=0}=-D_\tau(X_t\cdot\nu_t)=-D_\tau\bigl([\pa_{\nu_t}w_t]\bigr)  \qquad\text{on $\pa E_t$,} 
\eeq
where the last equality follows again by \eqref{xtnu}.
Moreover, by differentiating with respect to $s$ the identity $D\nu_{t+s}[\nu_{t+s}]=0$, we get $D\dot{\nu_t}[\nu_t]+D\nu_t[\dot{\nu_t}]=0$. Multiplying the latter equality  by $\nu_t$ and recalling that $D\nu_t$ is symmetric we get 
$D\dot{\nu_t}[\nu_t]\cdot \nu_t=-D\nu_t[\nu_t]\cdot \dot{\nu_t}=0$. 
In turn, this implies that 
\beq\label{divtdiv}
\Div_{\tau}\dot{\nu_t}=\Div \dot{\nu_t} \qquad\text{on $\pa E_t$.}
 \eeq
Also,
\beq\label{CMM3}
\begin{split} 
 \frac{\partial }{\partial s} (H_{{t+s}}\circ \Phi_s)\Bigl|_{s=0} & \stackrel{\phantom{\eqref{nupunto}}}{=}   \dot{H}_t + DH_t\cdot X_t= \\
 &\stackrel{\eqref{divtdiv}}{=}\Div_\tau \dot{\nu}_t+ \partial_{\nu} H_t (X_t\cdot \nu_t)+ D_\tau H_t\cdot X_t\\
 &\stackrel{\eqref{dnuacca}}{=}\Div_\tau \dot{\nu}_t-  |B_t|^2 [\partial_{\nu_t} w_t] + D_\tau H_t\cdot X_t\\
&\stackrel{\eqref{nupunto}}{=}- \Delta_\tau [\partial_{\nu_t} w_t] -  |B_t|^2 [\partial_{\nu_t} w_t] + D_\tau H_t\cdot X_t
  \end{split} 
\eeq
We can now compute
\begin{equation}
\label{long calculation}
\begin{split}
&\frac{d}{ds}\left(\frac{1}{2}   \int_{E_{t+s}} |D w_{t+s}|^2\, dx \right) \Bigl|_{s=0} =   \frac{d}{ds}  \left(\frac{1}{2}   \int_{E_{t}} |(D w_{t+s})\circ \Phi_s|^2 \, J \Phi_s \, dx \right) \Bigl|_{s=0} \\
&= \frac{1}{2}   \int_{E_{t}} |D w_t |^2 \diver X_t \, d x+   \int_{E_{t}}   Dw_t\cdot \left(  D^2 w_t [X_t]  +  D \dot{w}_t \right)\, dx \\
&= \frac{1}{2}   \int_{E_{t}} \diver( |D w_t |^2  X_t) \, dx+  \int_{E_{t}}    D w_t\cdot  D \dot{w}_t  dx \\
&= \frac{1}{2}   \int_{\partial E_{t}}  |D w^-_t |^2  X_t \cdot \nu_t\, d \Ha^{2}   + \int_{\partial E_{t}}   \dot{w}^-_t  \partial_{\nu_t} w^-_t \, d \Ha^{2}  \\
&= \frac{1}{2}   \int_{\partial E_{t}}  |D w^-_t |^2  [ \partial_{\nu_t} w_t ] \, d \Ha^{2}  + \int_{\partial E_{t}}   \dot{w}^-_t  \partial_{\nu_t} w^-_t \, d \Ha^{2}\,.
\end{split}
\end{equation}
In order to write $\dot{w}^-_t$ explicitly we use
\[
w_{t+s}^-= H_{t+s} + 4\gamma\, v_{{t+s}}  \qquad \text{on } \, \partial E_{t+s}\,,
\] 
which in turn is equivalent to 
\[
w^-_{t+s}\circ \Phi_s= H_{t+s}\circ\Phi_s + 4\gamma\, v_{t+s}\circ\Phi_s \qquad \text{on } \, \partial E_{t}.
\]
By differentiating the above identity with respect to $s$ at $s=0$, we get 
\[
\dot{w}^-_t +  Dw^-_t\cdot X_t = \dot{H}_t + DH_t \cdot X_t + 4\gamma \dot{v}_t + 4\gamma\,  Dv_t\cdot  X_t  \qquad \text{on } \, \partial E_{t}.
\]
We  now use \eqref{CMM3} (and of course \eqref{xtnu}) to get
\beq\label{w dot}
\begin{split}
\dot{w}_t^- =&  -(\partial_{\nu_t} w^-_t) [\partial_{\nu_t} w_t]   - \Delta_\tau [\partial_{\nu_t} w_t] -  |B_t|^2 [\partial_{\nu_t} w_t] \\
&+ 4\gamma\, \dot{v}_t + 4\gamma\, \partial_{\nu_t} v_t   [\partial_{\nu_t} w_t]+ D_\tau (H_t+ 4\gamma\, v_t- w_t )\cdot X_t \\
= &  -(\partial_{\nu_t} w^-_t) [\partial_{\nu_t} w_t]   - \Delta_\tau [\partial_{\nu_t} w_t] -  |B_t|^2 [\partial_{\nu_t} w_t] \\
&+ 4\gamma\, \dot{v}_t + 4\gamma\, \partial_{\nu_t} v_t   [\partial_{\nu_t} w_t] \qquad \text{on } \, \partial E_{t}\,,
\end{split}
\eeq
where in the last equality we have used the fact that $w_t=H_t+ 4\gamma\, v_t$ on $\pa E_t$.
Therefore from \eqref{dot v}, \eqref{long calculation} and \eqref{w dot} we get 
\begin{equation} \label{interior}
\begin{split}
\frac{d}{dt} \left(\frac{1}{2} \int_{E_t} |D w_t|^2\, dx \right) = &-\int_{\partial E_t} \partial_{\nu_t} w^-_t\,   \Delta_\tau [\partial_{\nu_t} w_t]\, d \Ha^{2} - \int_{\partial E_t} |B_t|^2 \, \partial_{\nu_t} w_t^-\,  [\partial_{\nu_t} w_t]  \, d \Ha^{2} \\
&+8\gamma   \int_{\partial E_t}\int_{\partial E_t} G_{\T^3}(x,y) \,\partial_{\nu_t} w_t^-(x)\, [\partial_{\nu_t} w_t(y)] \, d \Ha^{2}(y) d \Ha^{2}(x) \\
&+4\gamma  \int_{\partial E_t} \partial_{\nu_t} v_t\,\partial_{\nu_t} w_t^-\,  [\partial_{\nu_t} w_t]  \, d \Ha^{2} \\
&+\frac{1}{2}   \int_{\partial E_{t}}  |D w_t^- |^2  [\partial_{\nu_t} w_t]  \, d \Ha^{2} - \int_{\partial E_t} (\partial_{\nu_t} w_t^-)^2 [\partial_{\nu_t} w_t]   \, d \Ha^{2} .
\end{split}
\end{equation}

The analogous calculations  in $\T^3 \setminus E_t$ yield
\begin{equation} \label{exterior}
\begin{split}
\frac{d}{dt} \left(\frac{1}{2} \int_{\T^3 \setminus E_t} |D w_t|^2\, dx \right) &= \int_{\partial E_t} \partial_{\nu_t} w_t^+\,   \Delta_\tau [\partial_{\nu_t} w_t]\, d \Ha^{2} + \int_{\partial E_t} |B_t|^2  \partial_{\nu_t} w_t^+\,  [\partial_{\nu_t} w_t]  \, d \Ha^{2} \\
&-8\gamma   \int_{\partial E_t}\int_{\partial E_t} G_{\T^3}(x,y)\, \partial_{\nu_t} w_t^+(x)\, [\partial_{\nu_t} w_t(y)] \, d \Ha^{2}(y) d \Ha^{2}(x) \\
&-4\gamma  \int_{\partial E_t} \partial_{\nu_t} v_t\, \partial_{\nu_t} w_t^+\,  [\partial_{\nu_t} w_t]  \, d \Ha^{2} \\
&-\frac{1}{2}   \int_{\partial E_{t}}  |D w^+_t |^2  [\partial_{\nu_t} w_t]  \, d \Ha^{2} + \int_{\partial E_t} (\partial_{\nu_t} w_t^+)^2[\partial_{\nu_t} w_t]   \, d \Ha^{2} .
\end{split}
\end{equation}

Combining \eqref{interior} and \eqref{exterior},    integrating by   parts, and recalling \eqref{J2} we get
\[
\begin{split}
\frac{d}{dt} \left(\frac{1}{2} \int_{\T^3} |D w_t|^2\, dx \right) = &-  \pa^2J(E_t)\left[[\pa_{\nu_t}w_t\vphantom{^{^4}}]\right]
+\int_{\partial E_t}  \left((\partial_{\nu_t} w^+_t)^2 - (\partial_{\nu_t} w_t^-)^2  \right)[\partial_{\nu_t} w_t]   \, d \Ha^{2} \\
 &-\frac{1}{2}   \int_{\partial E_{t}}  (|D w^+_t |^2  - |D w^-_t |^2 ) [\partial_{\nu_t} w_t]  \, d \Ha^{2}.
\end{split}
\]
The result follows from the identity
\[
|D w^+_t |^2  - |D w^-_t |^2  = (\partial_{\nu_t} w^+_t)^2 - (\partial_{\nu_t} w_t^-)^2  = (\partial_{\nu_t} w_t^+ + \partial_{\nu_t} w_t^- ) [\partial_{\nu_t} w_t]  .
\]

\end{proof}

We now prove Proposition~\ref{harmonic estimates}.

\begin{proof}[Proof of Proposition~\ref{harmonic estimates}] To simplify the notation, throughout the proof we write $\nu$ instead of $\nu_E$.

\textbf{Proof of (i)}:  
  Observe that we may write $u$ as 
\[
u(x) = \int_{\partial E} G_{\T^3}(x,y)f(y) \, d\Ha^{2}(y).
\] 
Note that  $G_{\T^3}(x,y) = h(x-y) + r(x-y)$ where  $h$ is one-periodic, smooth away from $0$ and $h(t)=\frac{1}{4\pi|t|}$ in a neighborhood of $0$, while $r$ is smooth and one-periodic.
  The conclusion then follows since
 for $v(x):= \int_{\partial E} \frac{f(y)}{|x-y|} \, d\Ha^{2}(y)$ it holds
\[
 \|v\|_{L^p(\partial E)} \leq C  \|f\|_{L^p(\partial E)}.
\]

\textbf{Proof of (ii)}: Here we  adapt the proof of \cite{JK} to the periodic setting.
First observe that since $u$ is harmonic in $E\subset\T^3$ we have 
\begin{equation}
\label{harmonic lemma 12}
\diver\left(  2(Du\cdot x)  Du -|Du|^2x + u Du  \right) = 0.
\end{equation}
 Moreover, by the $C^{1,\alpha}$-regularity of $\partial E$ there exist $r>0$, $C_0$ and $N$, depending on the $C^{1,\alpha}$ bounds on $\partial E$,  such that 
 we may cover $\partial E$ with at most $N$ balls $B_r(x_k)$ such that, up to a translation,
\begin{equation}
\label{harmonic lemma 1}
\frac{1}{C_0} \leq  x \cdot \nu(x)  \leq C_0 \qquad\text{ for $x\in \pa E \cap B_{2r}(x_k)$.}
\end{equation}
  Therefore if $0 \leq \vphi_k \leq 1$ is a smooth function with compact support in $B_{2r}(x_k)$ such that $\vphi_k \equiv 1$
 in $B_r(x_k)$ and $|D \vphi_k| \leq C/r$, by integrating 
\[
\diver\left( \vphi_k \left(2(Du\cdot x)  Du -|Du|^2x +  u Du \right)  \right) 
\]
over $E$ and using \eqref{harmonic lemma 12} we easily get
\[
\begin{split}
\int_{\partial E} &2\vphi_k |\partial_\nu u|^2   (x \cdot \nu)  -  \vphi_k | D_\tau  u|^2   (x \cdot \nu)  \, d \Ha^{2}  \\
&=  - \int_{\partial E}  \vphi_k  u  \partial_\nu u  \, d \Ha^{2}- 2 \int_{\partial E}  \vphi_k  ( D_\tau  u \cdot x)  \partial_\nu u   \, d \Ha^{2} \\
&\,\,\,\,\,\,+\int_E D\vphi_k \cdot \left( 2(Du\cdot x)   Du -|Du|^2x +  u Du  \right) \, dx.
\end{split}
\]
This implies using the Poincar\'e inequality on the torus (recall that $u$ has zero average) and \eqref{harmonic lemma 1}
\[
\begin{split}
\int_{\partial E \cap B_r(x_k)} |\partial_\nu u|^2\, d \Ha^{2}   &\leq C \int_{\partial E} (u^2 + | D_\tau    u|^2) \, d \Ha^{2} + C\int_{\T^3} (u^2 +|Du|^2) \, dx\\
&\leq  C \int_{\partial E} (u^2 + | D_\tau    u|^2) \, d \Ha^{2} + C\int_{\T^3} |Du|^2 \, dx.
\end{split}
\]
Adding up all the estimates and repeating the argument for $\T^3 \setminus E$ we get
\[
\int_{\partial E } (|\partial_\nu u^-|^2 + |\partial_\nu u^+|^2)\, d \Ha^{2}  \leq  C \int_{\partial E} (u^2 + | D_\tau    u|^2) \, d \Ha^{2} + C\int_{\T^3} |Du|^2 \, dx.
\]
The result follows by observing that 
\[
\int_{\T^3} |Du|^2 \, dx = \int_{\partial E} u (\partial_\nu u^- - \partial_\nu u^+) \, d \Ha^{2}. 
\]

\textbf{Proof of (iii)}: The result would follow from the boundary estimates on $C^1$-domains established in \cite{FJR}. However, it turns out that in the case of $C^{1,\alpha}$-domains the argument can be greatly simplified, as shown in the following.
 
Let us define
\[
Kf(x) :=  \int_{\partial E} D_x G_{\T^3}(x,y) \cdot \nu(x) f(y) \, d\Ha^{2}(y)\,.
\]
We first show that the above integral is defined for every $x \in \partial E$ and  that 
\begin{equation}
\label{flow lemma 1}
\|Kf\|_{L^p(\partial E)} \leq C \|f\|_{L^p(\partial E)}.
\end{equation}
By the decomposition recalled at the beginning of the proof we have $D_xG_{\T^3}(x,y) = D_xh(x-y) + D_xr(x-y)$, where
$D_xh(x-y)=-\frac{1}{4\pi}\frac{x-y}{|x-y|^3}$ in a neighborhood of the origin and $D_xr(x-y)$ is smooth. 
Thus, by a standard  partition of unity argument we may localize the estimate and reduce to show that if $\varphi \in C^{1,\alpha}(\R^{2})$  and $U \subset \R^{2}$ is a bounded domain setting
$\Gamma := \{(x', \varphi(x')) \, : \, x' \in U\}$ and 
\[
Tf(x) := \int_\Gamma \frac{ (x-y)\cdot \nu(x) }{|x-y|^{3}} f(y) \, d\Ha^{2}(y) \quad x\in \Gamma,
\] 
where  $\nu$ is the upper normal to $\Gamma$, then $Tf(x)$ is well defined at every $x \in \Gamma$ and 
\[
\|Tf\|_{L^p(\Gamma)} \leq C \|f\|_{L^p(\Gamma)}.
\]
To show this we observe that we may write 
\[
Tf(x) := \int_U \frac{ \vphi(x') - \vphi(y')  -  D\vphi(x')\cdot( x'-y') }{(|x'-y'|^2+ (\vphi(x') -\vphi(y'))^2  )^{3/2}} f(y', \vphi(y')) \, dy'.
\] 
Therefore
\[
\begin{split}
|Tf(x)| &\leq C \int_U  \frac{ |x'-y'|^{1+ \alpha}}{(|x'-y'|^2+ (\vphi(x') -\vphi(y'))^2  )^{3/2}}|f(y', \vphi(y'))| \, dy' \\
&\leq C \int_U  \frac{ |f(y', \vphi(y'))|}{|x'-y'|^{2-\alpha}} \, dy'.
\end{split}
\] 
Thus the estimate \eqref{flow lemma 1} follows from a standard convolution estimate. 

For $x \in  E$  we have
\[
Du(x) = \int_{\partial E} D_xG_{\T^3}(x,y)f(y) \, d\Ha^{2}(y).
\]

Therefore for $x \in \partial E$  it holds
\[
Du(x - t\nu(x)) \cdot  \nu(x) = \int_{\partial E} D_xG_{\T^3}(x -t \nu(x),y) \cdot \nu(x)f(y) \, d\Ha^{2}(y).
\]
We claim that   
\begin{equation}
\label{flow lemma 2}
\lim_{t \to 0+} Du(x - t\nu(x)) \cdot  \nu(x) = Kf(x) +\frac{1}{2}f(x) 
\end{equation}
for every $x \in \partial E$.
Then the lemma follows from \eqref{flow lemma 1} and \eqref{flow lemma 2}. 

To show \eqref{flow lemma 2} we first recall  that for $z \in E$  and for $x \in \partial E$ it holds  
\begin{align}
&\int_{\partial E} D_xG_{\T^3}(z,y) \cdot \nu(y) \, d\Ha^{2}(y)  =   1  \quad \text{and} \nonumber\\
&  \int_{\partial E} D_xG_{\T^3}(x ,y) \cdot \nu(y) \, d\Ha^{2}(y)  =   \frac{1}{2}.  \label{flow lemma divergence} 
\end{align}
Therefore, we may write
\begin{equation} \label{flow lema long}
\begin{split}
 Du(x - t\nu(x)) \cdot  \nu(x) = &\int_{\partial E} D_xG_{\T^3}(x - t\nu(x),y) \cdot \nu(x) (f(y)-f(x)) \, d\Ha^{2}(y) \\
&+ f(x) \int_{\partial E} D_xG_{\T^3}(x - t\nu(x),y) \cdot (\nu(x)-\nu(y)) \, d\Ha^{2}(y) + f(x).
\end{split}
\end{equation}
Let us now prove that 
\begin{align*}
&\lim_{t \to 0} \int_{\partial E} D_xG_{\T^3}(x- t\nu(x),y) \cdot \nu(x) (f(y)-f(x)) \, d\Ha^{2}(y)\\
&\qquad\qquad\qquad\qquad\qquad = \int_{\partial E} D_xG_{\T^3}(x,y) \cdot \nu(x) (f(y)-f(x)) \, d\Ha^{2}(y). 
\end{align*}
To establish this, first observe that since $\pa E$ is $C^1$ then for $|t|$ sufficiently small we have
\beq\label{vesa1}
|x-y-t\nu(x)|\geq\frac12|x-y|\qquad\text{for all $y\in\pa E$}\,.
\end{equation}
Then, in view of  the decomposition of $D_xG$ recalled before, it is enough  show that 
\[
\begin{split}
\lim_{t \to 0} &\int_{\partial E} \frac{(x -y - t\nu(x)) \cdot \nu(x)}{|x-y-t\nu(x)|^3} (f(y)-f(x)) \, d\Ha^{2}(y)  \\
&= \int_{\partial E} \frac{(x -y) \cdot \nu(x)}{|x-y|^3} (f(y)-f(x)) \, d\Ha^{2}(y)\,,
\end{split}
\]
which follows from the Dominated Convergence Theorem, after observing that due to the  $\alpha$-H\"older continuity of $f$ and to \eqref{vesa1}, the absolute value of both integrands can be estimated from above by  $C/|x-y|^{2-\alpha}$ for some constant $C>0$.

Hence \eqref{flow lemma 2} follows by letting $t \to 0$ in \eqref{flow lema long} and recalling \eqref{flow lemma divergence}.

\textbf{Proof of (iv)}: Fix $p>2$ and $\beta\in (0, \frac{p-2}{p})$. As before, due to the properties of the Green's function it is sufficient to establish the statement for the 
function 
$$
 v(x):=\int_{\pa E}\frac{f(y)}{|x-y|}\, d\Ha^2(y)\,.
$$
For $x_1$, $x_2\in \pa E$ we have
$$
|v(x_1)-v(x_2)|\leq \int_{\pa E}|f(y)|\frac{\big||x_1-y|-|x_2-y|\big|}{|x_1-y|\, |x_2-y|}\, d\Ha^2(y)\,.
$$
In turn, by an elementary inequality, we have
$$
\frac{\big||x_1-y|-|x_2-y|\big|}{|x_1-y|\, |x_2-y|}\leq C(\beta)\frac{\big||x_1-y|^{1-\beta}+|x_2-y|^{1-\beta}\big|}{|x_1-y|\, |x_2-y|}|x_1-x_2|^\beta\,.
$$
Thus, by H\"older inequality we have
\begin{align*}
|v(x_1)-v(x_2)| &\leq C(\beta) \int_{\pa E}|f(y)|\frac{\big||x_1-y|^{1-\beta}+|x_2-y|^{1-\beta}\big|}{|x_1-y|\, |x_2-y|}\, d\Ha^2(y)\,\, |x_1-x_2|^\beta \\
&\leq C'(\beta)\|f\|_{L^p} |x_1-x_2|^\beta\,, 
\end{align*}
where we set 
$$
C'(\beta):=C(\beta)\biggl(2\sup_{z_1,\, z_2\in\pa E} \int_{\pa E}\frac{1}{|z_1-y|^{\beta p'}\, |z_2-y|^{p'}}\, d\Ha^2(y)\biggr)^{\frac{1}{p'}}\,.
$$
\textbf{Proof of (v)}:  We start by observing that
\[
\| f\|_{L^2(\partial E)} \leq C \|f\|_{H^1(\partial E)}^{^{\frac{1}{2}}}  \|f\|_{H^{-1}(\partial E)}^{\frac12},
\]
where $C$ is a constant depending only on the $C^{1,\alpha}$ bounds on $\partial E$. If $p>2$ we have also, see Lemma~\ref{interpolation},
\[
\| f\|_{L^p(\partial E)} \leq C \|f\|_{H^1(\partial E)}^{^{\frac{p-2}{p}}}  \|f\|_{L^2(\partial E)}^{\frac2p}.
\]
Therefore, by combining the two previous inequalities we get that for $p\geq2$
\[
\| f\|_{L^p(\partial E)} \leq C \|f\|_{H^1(\partial E)}^{^{\frac{p-1}{p}}}  \|f\|_{H^{-1}(\partial E)}^{\frac1p}.
\]
 Hence the claim follows once we show
\[
\|f \|_{H^{-1}(\partial E)} \leq C  \|u\|_{L^2(\partial E)}.
\]
Let us fix  $\vphi \in H^1(\partial E)$ and with abuse of notation denote its harmonic extension to $\T^3$  by $\vphi$. Then by integrating  by parts twice  and by  (ii) we get 
\[
\begin{split}
\int_{\partial E} \vphi f \, d \Ha^2 &= - \int_{\partial E} u [\partial_\nu \vphi]  \, d \Ha^2   \leq \|u \|_{L^2(\partial E)} \|[\partial_\nu \vphi] \|_{L^2(\partial E)} \\
&\leq   \|u \|_{L^2(\partial E)} \left( \|\partial_\nu \vphi^+ \|_{L^2(\partial E)} + \|\partial_\nu \vphi^- \|_{L^2(\partial E)} \right) \\
&\leq C \|u\|_{L^2(\partial E)} \|\vphi \|_{H^1(\partial E)}.
\end{split}
\] 
Therefore
\[
\| f\|_{H^{-1}(\partial E)}  = \sup_{\|\vphi\|_{H^1(\partial E)}\leq 1} \int_{\partial E} \vphi f \, d \Ha^2  \leq C \|u \|_{L^2(\partial E)}.
\]

\end{proof}

We now prove Lemma~\ref{w52conv}. Before that we recall that for $E\subset\T^3$  the $H^{\frac12}(\pa E)$ {\it Gagliardo seminorm} of a function $f\in L^2(\pa E)$ is defined by setting
$$
[f]_{\frac12,\pa E}^2:=\int_{\pa E}\,d\Ha^2(x)\int_{\pa E}\frac{|f(x)-f(y)|^2}{|x-y|^3}\,d\Ha^2(y)\,.
$$
Starting from this definition and using a standard partition of unity argument in order to straighten the boundary of $E$ locally, the reader may reconstruct the proof of the following technical lemma.
\begin{lemma}\label{nicola1}
Let $E\subset\T^3$ be an open set of class $C^{1,\alpha}$ for some $\alpha\in(0,1)$. For every $\gamma\in[0,\frac12)$, there exists a constant $C$ depending only on $\gamma$ and on the $C^{1,\alpha}$ bounds on $\partial E$ such that if $f\in H^{\frac12}(\pa E)$ and $g\in W^{1,4}(\pa E)$ then
$$
[fg]_{\frac12}\leq\big([f]_{\frac12}\|g\|_{L^\infty}+\|f\|_{L^{\frac{4}{1+\gamma}}}\|g\|_{L^\infty}^\gamma\|D_\tau g\|_{L^4}^{1-\gamma}\big)\,.
$$
\end{lemma}
Next lemma is probably well known to the expert, but we give its proof for reader's convenience
\begin{lemma}l\label{noia}
Let $F,U$ be as in Lemma~\ref{w52conv}. Let $E$ be a set in $\mathfrak{h}^{1,\alpha}_M(F, U)$, for some $\alpha>0$. If $H_{\pa E}\in H^{\frac12}(\pa E)$, then $E$ is of class $W^{\frac52,2}$ and
$$
\|\psi_E\|_{W^{\frac52,2}(\pa F)}\leq C(M)\big(1+\|H_{\pa E}\|_{H^{\frac12}(\pa E)}^2\big)\,,
$$
where $\psi_E$ is defined as in \eqref{front}.
\end{lemma}
\begin{proof} We assume without loss of generality that $\psi_E$ is smooth. To simplify the notation  we will drop the subscript from $\psi_E$ and $H_{\pa E}$. Fix $\e>0$. By straightening locally the boundary of $F$, we may reduce to the case where the function $\psi$ is defined in a disk $B'\subset\R^2$ and $\|\psi\|_{C^1(B')}\leq\e$. Fix a cut-off function $\varphi$ with compact support in $B'$. Then
\beq\label{noia1}
\Delta(\varphi\psi)-\frac{D^2(\varphi\psi)D\psi D\psi}{1+|D\psi|^2}=\varphi H\sqrt{1+|D\psi|^2}+R(x,\psi,D\psi)\,,
\eeq
where the remainder term $R$ is a smooth Lipschitz function. Then, using Lemma~\ref{nicola1} with $\gamma=0$ and recalling that $\|\psi\|_{C^1}\leq\e$, we estimate
\begin{align*}
[\Delta(\varphi\psi)]_{\frac12}\leq C(M)\big(\e^2[D^2(\varphi\psi)]_{\frac12}+[H]_{\frac12}(1\!+\!\|D\psi\|_{L^\infty})+\|H\|_{L^4}(1\!+\!\|\psi\|_{W^{2,4}})+1+\|\psi\|_{W^{2,4}}\big).
\end{align*}
Observe that by Calder\'on-Zygmund estimates $\|\psi\|_{W^{2,4}(B')}\leq C(M)(1+\|H\|_{L^4(\pa E)})$. Moreover, a simple integration by part argument shows that if $u$ is a smooth function with compact support in $\R^2$ then
$$
[\Delta u]_{\frac12,\R^2}=[D^2u]_{\frac12,\R^2}\,.
$$
Thus, choosing $\e$ sufficiently small, we may conclude that
$$
[D^2(\varphi\psi)]_{\frac12}\leq C(M)\big(1+[H]_{\frac12,\pa E}+\|H\|_{L^4(\pa E)}^2\big)\leq C(M)\big(1+\|H\|_{H^{\frac12}(\pa E)}^2\big)\,.
$$
From this estimate the conclusion follows.
\end{proof}
\begin{proof}[Proof of Lemma~\ref{w52conv}] {\bf Step 1.}
Throughout the proof we write $w_n$, $H_n$, and $v_n$ instead of $w_{E_n}$, $H_{\pa E_n}$, and $v_{E_n}$, respectively. Moreover we denote by $\hat w_n$ the average of $w_n$ in $\T^3$ and we set $\tilde w_n=\medintinrigo_{\pa E_n}w_n\,d\Ha^2$ and $\tilde H_n=\medintinrigo_{\pa E_n}H_n\,d\Ha^2$.
First, recall that 
\beq\label{w521}
w_n=H_n+4\gamma v_n \quad\text{on }\pa E_n\qquad\text{and}\qquad \sup_n\|v_n\|_{C^{1,\alpha}(\T^3)}<+\infty\,.
\eeq
The last bound follows from  standard elliptic estimates. Moreover, from the trace inequality
\beq\label{w522}
\|w_n-\tilde w_n\|^2_{H^{\frac12}(\pa E_n)}\leq\|w_n-\hat w_n\|^2_{H^{\frac12}(\pa E_n)}\leq C\int_{\T^3}|Dw_n|^2\, dx
\eeq
with $C$ depending only on the $C^1$-bounds on $\pa E_n$. 
We  claim that  
\beq\label{claim1111}
\sup_{n}\|H_n\|_{H^{\frac12}(\pa E_n)}<\infty.
\eeq
To see this note that by the uniform $C^1$-bounds on $\pa E_n$, we may find a fixed cylinder of the form  $C:=B'\times(-L,L)$, with $B'\subset\R^{2}$  a ball centered at the origin, and functions $f_n$, with
\beq\label{w523}
\sup_{n}\|f_n\|_{C^1(\overline B')}<+\infty\,,
\eeq
such that  $\pa E_n\cap C=\{(x',x_n)\in B'\times(-L,L):\, x_n= f_n(x')\}$ with respect to a suitable  coordinate frame (depending on $n$).  
Thus we have
\begin{align*}
\int_{B'}(H_n-\tilde H_n)\, dx'+ \tilde H_n|B'| &= \int_{B'}\Div\biggl(\frac{\nabla_{x'} f_n}{\sqrt{1+|\nabla_{x'} f_n|^2}}\biggr)\, dx'\\
&=\int_{\pa B'}\frac{\nabla_{x'} f_n}{\sqrt{1+|\nabla_{x'} f_n|^2}}\cdot \frac{x'}{|x'|}\, d\mathcal{H}^{1}\,.
\end{align*}
Hence, recalling \eqref{w523} and the fact that $\|H_n- \tilde H_n\|_{H^{\frac12}(\pa E_n)}$ is bounded thanks to \eqref{w521} and \eqref{w522}, we get that ${\tilde H_n}$ are bounded. Therefore the claim \eqref{claim1111} follows.

By applying the Sobolev embedding theorem on each connected component of $\partial F$ we have that $\|H_n\|_{L^4(E_n)}$ is bounded. This fact, together with the uniform $C^1$ bounds on $\pa E_n$ implies that if we write 
$$
\pa E_n:=\{x+\psi_n(x):\, x\in \pa F\}\,,
$$ 
then $\sup_{n}\|\psi_n\|_{W^{2,4}(\pa F)}<+\infty$.  This follows by standard elliptic estimates, see \cite[Lemma~7.2 and Remark~7.3]{AFM}. Thus,    up to a (not relabeled) subsequence, there exists a set $F'\in \mathfrak{C}^1_M(F, U)$ such that 
$$
\psi_n\to \psi_{F'} \text{ in $C^{1,\alpha}(\pa F)$}\quad\text{and}\quad v_n\to v_{F'}\text{ in $C^{1,\beta}(\T^3)$ } \quad\text{for all $\alpha\in (0,\tfrac12)$ and $\beta\in (0,1)$.}
$$
From  \eqref{claim1111} and Lemma~\ref{noia} we have that the functions $\psi_n$ are bounded in $W^{\frac52,2}(\pa F)$. Hence the first part of the statement follows.
 \par\noindent
 {\bf Step 2.} For the second part we first observe that if 
 $$
 \int_{\T^3}|Dw_n|^2\, dx\to 0
 $$
 then the above arguments yield the existence of $\lambda\in \R$ and a  (not relabelled) subsequence such that $w_n\big(\cdot + \psi_n(\cdot)\nu_F(\cdot)\big)\to \lambda$  in $H^{\frac12}(\pa F)$. In turn,
$$
 H_n\big(\cdot + \psi_n(\cdot)\nu_F(\cdot)\big)\to \lambda-4\gamma v_{F'}\big(\cdot + \psi_{F'}(\cdot)\nu_F(\cdot)\big)=
 H_{\pa F'}\big(\cdot + \psi_{F'}(\cdot)\nu_F(\cdot)\big)\qquad\text{in }H^{\frac12}(\pa F)\,.
 $$
To conclude the proof we need to show that  $\psi_n$ converge to $\psi:=\psi_{F'}$ in $W^{\frac52,2}(\pa F)$. To this aim, fix $\e>0$. By straightening locally the boundary of $F$, we may always reduce to the case where the functions $\psi_n$ are defined on a disk $B'\subset\R^2$, are bounded in $W^{\frac52,2}(B')$, converge in $W^{2,p}(B')$ for all $p\in[1,4)$ to  $\psi\in W^{\frac52,2}(B')$ and 
 $\|D\psi\|_{L^\infty(B')}\leq\e$. We fix a cut-off function $\varphi$ with compact support in $B'$ and we write
  \begin{align*}
  \frac{\Delta(\varphi\psi_n)}{\sqrt{1+|D\psi_n|^2}}-\frac{\Delta(\varphi\psi)}{\sqrt{1+|D\psi|^2}} &=(D^2(\varphi\psi_n)-D^2(\varphi\psi))\frac{D\psi D\psi}{(1+|D\psi|^2)^{\frac32}}\\
 &\qquad+ D^2(\varphi\psi_n)\biggl(\frac{D\psi_nD\psi_n}{(1+|D\psi_n|^2)^{\frac32}}-\frac{D\psi D\psi}{(1+|D\psi|^2)^{\frac32}}\biggr)\\
 &\qquad+\varphi(H_n-H)+R(x,\psi_n,D\psi_n)-R(x,\psi,D\psi)\,,
   \end{align*}
   where the remainder term is $R$ is similar to the one in \eqref{noia1}.
  Then,  using Lemma~\ref{nicola1} with $\gamma\in(0,\frac12)$, an argument similar to the one of the proof of Lemma~\ref{noia} shows that
 \begin{align*}
 &\bigg[\frac{\Delta(\varphi\psi_n)}{\sqrt{1+|D\psi_n|^2}}-\frac{\Delta(\varphi\psi)}{\sqrt{1+|D\psi|^2}}  \bigg]_{\frac12}\leq C(M)\big(\e^2[D^2(\varphi\psi_n)-D^2(\varphi\psi)]_{\frac12}\\
& \qquad+\|D^2(\varphi\psi_n)-D^2(\varphi\psi)\|_{L^{\frac{4}{1+\gamma}}}\|D\psi\|_{L^\infty}^\gamma\|D^2\psi\|_{L^4}^{1-\gamma}+
[D^2(\varphi\psi_n)]_{\frac12}\|D\psi_n-D\psi\|_{L^\infty}\\
& \qquad+\|D^2(\varphi\psi_n)\|_{L^{\frac{4}{1+\gamma}}}\|D\psi_n-D\psi\|_{L^\infty}^\gamma(\|D^2\psi_n\|_{L^4}+\|D^2\psi\|_{L^4})^{1-\gamma}\\
& \qquad
+\|H_n-H\|_{H^{\frac12}}+\|\psi_n-\psi\|_{W^{2,2}}\big)\,.
 \end{align*}
Using Lemma~\ref{nicola1} again to estimate $[\Delta(\varphi\psi_n)-\Delta(\varphi\psi)]_{\frac12}$ with the seminorm on the left hand side of the previous inequality and arguing as in the proof of Lemma~\ref{noia} we finally get
$$
[D^2(\varphi\psi_n)-D^2(\varphi\psi)]_{\frac12}\leq C(M)\big(\|\psi_n-\psi\|_{W^{2,\frac{4}{1+\gamma}}}
+\|D\psi_n-D\psi\|_{L^\infty}^\gamma+\|H_n-H\|_{H^{\frac12}}\big)\,,
$$
from which the conclusion follows.
\end{proof}
\subsection{The surface diffusion flow: proof of technical lemmas}\label{subsec:SD}
We start by providing the computations leading to the crucial energy identities of Lemma~\ref{calculationsbis}.

\begin{proof}[Proof of Lemma~\ref{calculationsbis}]
Let   $\Psi$, $\Psi_t$, $X_t$ be as in the proof of Lemma~\ref{calculations}, and note that by \eqref{SD} we have
\beq\label{xtnubis}
X_t\cdot \nu_t=\Delta_\tau H_t \qquad\text{ on $\pa E_t$.}
\eeq
Fix $t\in (0,T)$,  and as in Lemma~\ref{calculations} set $\Phi_s:=\Psi_{t+s}\circ \Psi_t^{-1}$,  so that  $(\Phi)_{s\in (-t, T-t)}$ is an admissible  one-parameter family  of diffeomorphisms according to Definition~\ref{def:admissibleX}.  Then, by Theorem~\ref{th:12var} we get
\begin{align*}
\frac{d}{dt}J(E_t)&=\frac{d}{ds}J(\Phi_s(E_t))_{\bigl|_{s=0}}\\
&=\int_{\pa E_t}H_{t}X_t\cdot \nu_t\, d\Ha^{2}=
\int_{\pa E_t}H_{t}\Delta_\tau H_t\, d\Ha^{2}=-\int_{\pa E_t}|D_\tau H_{t}|^2\, d\Ha^{2}\,. 
\end{align*}
This establishes \eqref{der of Jbis}.
 Let us fix a time $t>0$. To continue we  observe that, by redefining the velocity field if needed (in a time interval centered at $t$), we may assume that $X_t$ has only a normal component on $\pa E_t$; that is, 
 \beq\label{normonly}
 X_t=(X_t\cdot \nu_t)\nu_t\qquad\text{on $\pa E_t$.}
 \eeq 
Recall that all the geometric quantities can be extended in a neighborhood of $\pa E_t$ by means of the gradient of the signed distance function from $E_t$ (see the proof of Lemma~\ref{calculations}). Now, arguing as in \eqref{nupunto}, we have
\beq\label{nupuntobis}
\dot{\nu}_t=-D_\tau(X_t\cdot\nu_t)=-D_\tau\Delta_\tau H_t  \qquad\text{on $\pa E_t$,} 
\eeq
where the last equality follows again by \eqref{xtnubis}. In turn, using also \eqref{normonly} and \eqref{nupunto}
\begin{multline} \label{CMM3bis}
 \frac{\partial }{\partial s} (DH_{{t+s}}\circ \Phi_s)\Bigl|_{s=0}  
= D  \diver_\tau(\dot{\nu_t}) + D^2H_t [X_t] = - D (\Delta_\tau (\Delta_\tau H_t)) + (\Delta_\tau H_t) D^2H_t \nu_t    
\end{multline}
on $\pa E_t$.
Denoting by $D_{\tau_{t+s}}$ the tangential differential on $\pa E_{t+s}$ and by $J_\tau\Phi_s$ the  tangential Jacobian
of $\Phi_s$,   we have  
\beq\label{der1SD}
\begin{split}
&\frac{d}{ds} \left(\frac{1}{2}   \int_{\partial E_{t+s}} |D_\tau H_{t+s}|^2\, d \Ha^{2} \right)\Bigl|_{s=0} =   \frac{d}{ds} \left(\frac{1}{2}   \int_{\partial E_{t}} |D_{\tau_{t+s}} H_{t+s}|^2\circ\Phi_s J_\tau\Phi_s\, d \Ha^{2}\right) \Bigl|_{s=0} \\
&= \frac{1}{2}   \int_{\partial E_{t}} |D_\tau H_{t}|^2 \diver_\tau (\Delta_\tau H_t\, \nu_t )\, d \Ha^{2} +   \int_{\partial E_{t}} D_\tau H_{t} \cdot \frac{\pa }{\pa s}  \left( D_{\tau_{t+s}} H_{t+s}\circ \Phi_s \right)\Bigl|_{s=0}\, d \Ha^{2}.
\end{split}
\eeq
We write the last term  as
\[
D_{\tau_{t+s}} H_{t+s}\circ \Phi_s   = \left[ I - \nu_{t+s}\circ \Phi_s  \otimes \nu_{t+s}\circ \Phi_s   \right] DH_{t+s}\circ \Phi_s 
\]
and get by \eqref{normonly}, \eqref{dnuacca}, \eqref{nupuntobis} and \eqref{CMM3bis}
\begin{align}\label{nasty10}
&\quad\qquad\frac{\pa }{\pa s}  \left( D_{\tau_{t+s}} H_{t+s}\circ \Phi_s \right)\Bigl|_{s=0} =   (-  \dot{\nu}_t \otimes \nu_t - \nu_t \otimes  \dot{\nu}_t) DH_t + \left[ I - \nu_t \otimes \nu_t  \right]  \frac{\partial }{\partial t} (DH_{{t}}\circ \Phi_t)\nonumber\\
&\qquad\qquad\qquad=- |B_t|^2 D_\tau \Delta_\tau H_t -  DH_t\cdot \dot{\nu}_t\,  \nu_t -  D_\tau \Delta_\tau \Delta_\tau H_t  + \Delta_\tau H_t  \left[ I - \nu_t \otimes \nu_t  \right]D^2H_t \nu_t\,.
\end{align}
In order to calculate $D^2H_t \nu_t$ we differentiate the equation \eqref{dnuacca}  and get
\[
-D |B_t|^2 = D  (DH_t\cdot  \nu_t) = D^2H_t \nu_t + D\nu_t DH_t. 
\]
Therefore, since $B_t= D\nu_t$ and $B_t\nu_t = 0$ we get
\[
D^2H_t \nu_t = -D |B_t|^2 - B D_\tau H_t.
\]
Plugging the last identity in \eqref{nasty10} and using again \eqref{nupuntobis}, we may continue from \eqref{der1SD} to obtain
\beq\label{der2SD}
\begin{split}
&\frac{d}{ds}\left(\frac{1}{2}   \int_{\partial E_{t+s}} |D_\tau H_{t+s}|^2\, d \Ha^{2} \right) \Bigl|_{s=0} = \frac{1}{2}   \int_{\partial E_{t}}H_t  |D_\tau H_t|^2 \Delta_\tau H_t \, d \Ha^{2}  \\
&-  \int_{\partial E_{t}}|B_t|^2 \,  D_\tau H_t\cdot D_\tau\Delta_\tau H_t\, d \Ha^{2} -  \int_{\partial E_{t}}  D_\tau H_t\cdot D_\tau\Delta_\tau\Delta_\tau H_t\, d \Ha^{2}\\
&- \int_{\partial E_{t}} (\Delta_\tau H_t)  D_\tau |B_t|^2\cdot D_\tau H_t\, d \Ha^{2} -\int_{\partial E_t}  B[ D_\tau H_t] \Delta_\tau H_t \, d \Ha^{2}.
\end{split}
\eeq
Integrating the third term on the right-hand side by parts twice, we get
\[
-  \int_{\partial E_{t}}  D_\tau H_t\cdot D_\tau\Delta_\tau\Delta_\tau H_t\, d \Ha^{2}= - \int_{\partial E_t} |D_\tau\Delta_\tau H_t|^2\, d \Ha^{2}\,.
\]
Integrating the second last term on the right-hand side by parts once, we have 
\[
\begin{split}
&- \int_{\partial E_{t}} (\Delta_\tau H_t)  D_\tau |B_t|^2\cdot D_\tau H_t\, d \Ha^{2} \\
&=\int_{\partial E_{t}}|B_t|^2  D_\tau H_t\cdot  D_\tau\Delta_\tau H_t\, d \Ha^{2} 
+   \int_{\partial E_t} |B_t|^2|\Delta_\tau H_t|^2  \, d \Ha^{2}. 
\end{split}
\]
Plugging the last two identities into \eqref{der2SD} and recalling \eqref{J2} (with $\gamma=0$), the identity \eqref{der of DH} follows. 
\end{proof}

\begin{proof}[Proof of Lemma~\ref{laplacian}] 
In the following proof, in order to simplify the notation we drop the dependence on  $\pa E$ from all the geometric objects and the $L^p$ spaces involved.
Let us first show  
\begin{equation}
\label{laplacian 1}
\int_{\partial E} |D_\tau^2 f|^2 \, d \Ha^2 \leq C \int_{\partial E} | \Delta_\tau f|^2 \, d \Ha^2  + C \int_{\partial E} |B|^2|D_\tau f|^2  \, d \Ha^2.
\end{equation}
Indeed, recalling the following formula (see \cite[Eq.~(10.16)]{Giu})
\begin{equation}
\label{laplacian 2}
\delta_i \delta_j  = \delta_j \delta_i + (\nu_i \delta_j \nu_k - \nu_j \delta_i \nu_k )\delta_k
\end{equation}
and integrating by parts we get
\[
\begin{split}
&\int_{\partial E} |D_\tau^2 f|^2 \, d \Ha^2  =  \int_{\partial E} (\delta_i \delta_j f)  \, (\delta_i \delta_j f)    \, d \Ha^2 \\
&= \int_{\partial E} (\delta_i \delta_j f)  \, (\delta_j \delta_i f)    \, d \Ha^2   + \int_{\partial E} (\delta_i \delta_j f) (\nu_i \delta_j \nu_k - \nu_j \delta_i \nu_k )\delta_k f  \,   d \Ha^2\\
&= - \int_{\partial E} \delta_j f  \, (\delta_i \delta_j \delta_i f)    \, d \Ha^2 +  \int_{\partial E} H \nu_i \delta_j f  \, ( \delta_j \delta_i f)    \, d \Ha^2  + \int_{\partial E} (\delta_i \delta_j f) (\nu_i \delta_j \nu_k - \nu_j \delta_i \nu_k )\delta_k f  \,   d \Ha^2\\
&\leq  - \int_{\partial E} \delta_j f  \, (\delta_i \delta_j \delta_i f)    \, d \Ha^2 +  C \int_{\partial E} |B|\,|D_\tau f|\, |D_\tau^2 f|\,d \Ha^2.
\end{split}
\]
Using \eqref{laplacian 2} and  integrating by parts again, we obtain 
\[
\int_{\partial E} |D_\tau^2 f|^2 \, d \Ha^2 \leq \int_{\partial E}  (\delta_i \delta_i f)  \, (\delta_j \delta_j f)    \, d \Ha^2  \, d \Ha^2 +  C \int_{\partial E} |B|\, |D_\tau f|\, |D_\tau^2 f|\,d \Ha^2.
\]
The inequality \eqref{laplacian 1} follows  since $\Delta_\tau f = \delta_i \delta_i f$.

We estimate the  last term in \eqref{laplacian 1} by Lemma \ref{interpolation}: 
\[
\begin{split}
 \int_{\partial E} |B|^2|D_\tau f|^2  \, d \Ha^2 &\leq \|B\|_{L^4}^2   \|D_\tau  f\|_{L^4}^2 \\
&\leq  C \|B\|_{L^4}^2 \left( \|D_\tau^2 f\|_{L^2}  \|D_\tau  f\|_{L^2}  + \|D_\tau  f\|_{L^2}^2  \right)\,.
\end{split}
\]
Plugging in \eqref{laplacian 1} and  by an application of Young's inequality, we get 
\beq\label{aY}
\begin{split}
\|D_\tau^2 f\|_{L^2}^2 &\leq  C \left( \|\Delta_\tau f\|_{L^2}^2 +  \|D_\tau  f\|_{L^2}^2 ( \|B\|_{L^4}^2 +  \|B\|_{L^4}^4 ) \right) \\
&\leq  C \left( \|\Delta_\tau f\|_{L^2}^2 +  \|D_\tau  f\|_{L^2}^2 ( 1+  \|B\|_{L^4}^4 ) \right).
\end{split}
\eeq
Now, note that (with the same notation introduced in Lemma~\ref{interpolation})
\beq\label{poin}
\begin{split}
\|D_\tau  f\|_{L^2}^2  &= -\int_{\partial E}  f \Delta_\tau f \, d \Ha^2   = -  \int_{\pa E}(f- \bar f)  \Delta_\tau f \, d \Ha^2 \\
&\leq\|f- \bar f\|_{L^2}\|\Delta_\tau f \|_{L^2}\leq C \|D_\tau  f\|_{L^2} \|\Delta_\tau f \|_{L^2}\,.
\end{split}
\eeq
Note that in the second equality above we have used the fact that $\Delta_\tau f$ has zero average on each connected component of $\pa E$.
Thus,  from \eqref{aY} we deduce  
\[
\|D_\tau^2 f\|_{L^2}^2 \leq  C \|\Delta_\tau f\|_{L^2}^2 ( 1 + \|B\|_{L^4}^4 ).
\]
By a standard application of Calderon-Zygmund estimate we have
\[
 \|B\|_{L^4} \leq C(1+  \|H\|_{L^4}),
\] 
with $C$ depending only the  $C^{1}$-bounds on $\pa E$, and the conclusion follows.
\end{proof}

We now show the geometric interpolation used in the proof of Theorem~\ref{main thm 2}.

\begin{proof}[Proof of Lemma~\ref{nasty}] Also here to simplify the notation we drop the dependence on $\pa E$ both from the geometric objects and the $L^p$ spaces.
First  by H\"older's inequality 
\[
\int_{\partial E} |B| |D_\tau H|^2 |\Delta_\tau H| \, d \Ha^{2} \leq \|\Delta_\tau H\|_{L^3} \left(  \int_{\partial E} |B|^\frac{3}{2}|D_\tau H|^3 \, d \Ha^{2} \right)^{2/3}.
\]
By the Poincar\'e Inequality stated in Lemma~\ref{interpolation} we get 
\[
 \|\Delta_\tau H\|_{L^3} \leq C  \| D_\tau (\Delta_\tau H)\|_{L^2}.
\]
In turn, H\"older's inequality implies 
\[
 \left(  \int_{\partial E} |B|^\frac{3}{2}|D_\tau H|^3 \, d \Ha^{2} \right)^{2/3} \leq  \left(  \int_{\partial E}|D_\tau H|^{4} \, d \Ha^{2} \right)^{1/2}\left(  \int_{\partial E}|B|^{6} \, d \Ha^{2} \right)^{1/6}.
\]
Lemma  \ref{interpolation} yields
\[
\left(  \int_{\partial E}|D_\tau H|^{4} \, d \Ha^{2} \right)^{1/2} \leq C  \left( \|D_\tau^2 H\|_{L^2}  \|D_\tau  H\|_{L^2}  + \|D_\tau  H\|_{L^2}^2  \right).
\]
Combining all the inequalities above, we get 
\[
\int_{\partial E} |B| |D_\tau H|^2 |\Delta_\tau H| \, d \Ha^{2} \leq C  \| D_\tau (\Delta_\tau H)\|_{L^2}  \,  \|B\|_{L^6} \,   \|D_\tau H\|_{L^2} (\|D_\tau^2 H\|_{L^2}  + \|D_\tau  H\|_{L^2}). 
\]
By Lemma \ref{laplacian} and  \eqref{poin} (with $D_\tau H$ in place of $D_\tau f$), the right-hand side of the above inequality can be estimated from above  by 
\[
C \| D_\tau (\Delta_\tau H)\|_{L^2}  \, \|B\|_{L^6}  \,   \|\Delta_\tau H \|_{L^2} \,  \|D_\tau H\|_{L^2}  \, (1  + \| H\|_{L^4}^2).
\]
The conclusion follows from the  Poincar\'e Inequality
$$
\|\Delta_\tau H\|_{L^2} \leq C  \| D_\tau (\Delta_\tau H)\|_{L^2}.
$$
 and the  Calderon-Zygmund estimate
\[
 \|B\|_{L^6} \leq C(1+  \|H\|_{L^6})\,.
\] 
\end{proof}

We conclude with the proof of the geometric Poincar\'e Inequality stated in Lemma~\ref{lm:geopoinc}.

\begin{proof}[Proof of Lemma~\ref{lm:geopoinc}]
Since $\int_{\pa E}(H_{\pa E}-\overline H_{\pa E})\nu_E\, d\Ha^2=0$, we may apply 
Lemma~\ref{from AFM}, with $\e=1$ and $\vphi:=H_{\pa E}-\overline H_{\pa E}$, and recall \eqref{J2} (with $\gamma=0$) to obtain
\begin{multline*}
\sigma \int_{\pa E}|H_{\pa E}-\overline H_{\pa E}|^2\, d\Ha^2\\ \leq \int_{\pa E}|D_\tau H_{\pa E}|^2\, d\Ha^2-\int_{\pa E}|B_{\pa E}|^2|H_{\pa E}-\overline H_{\pa E}|^2\, d\Ha^2 \leq \int_{\pa E}|D_\tau H_{\pa E}|^2\, d\Ha^2\,.
\end{multline*}
The conclusion follows. 
\end{proof}

\section*{Acknowledgment}
The work of V. Julin was partially funded by the Academy of Finland grant 268393. The work of N. Fusco, V. Julin and M. Morini has been partially carried on at the University of Jyv\"{a}skyl\"{a} and supported by the FiDiPro project 2100002028.
The friendly atmosphere of the Mathematics and Statistics Department of  Jyv\"{a}skyl\"{a} is warmly acknowledged.

\end{document}